\documentclass[10pt,a4paper]{amsart}

\usepackage[austrian,english]{babel}
\usepackage[latin1]{inputenc}
\usepackage[a4paper]{geometry}

\usepackage{hyperref}
\usepackage{xparse}

\usepackage{amsmath}
\usepackage{amssymb}
\usepackage{mathtools}

\usepackage{tikz,tikz-3dplot}
\tikzset{every picture/.append style={scale=0.85}}
\tdplotsetmaincoords{60}{120}

\newtheorem{thm}{Theorem}

\newcommand{\bop}[1]{\boldsymbol{\operatorname{#1}}}
\newcommand{\abs}[1]{\left| #1 \right|}
\newcommand{\N}{\mathbb{N}}
\newcommand{\n}{\bop{n}}
\newcommand{\e}{\bop{e}}
\newcommand{\s}{\bop{s}}
\newcommand{\w}{\bop{w}}
\newcommand{\A}{\mathcal{A}}
\newcommand{\B}{\mathcal{B}}
\newcommand{\C}{\mathcal{C}}
\renewcommand{\O}{\mathcal{O}}
\newcommand{\I}{\mathcal{I}}
\newcommand{\tr}{\operatorname{tr}}
\newcommand{\Z}{\mathbb{Z}}

\begin{document}

\title{Building reverse plane partitions with rim-hook-shaped bricks}

\author{Robin Sulzgruber}
\thanks{Research supported by the Austrian Science Fund (FWF), grant S50-N15 in the framework of the Special Research Program ``Algorithmic and Enumerative Combinatorics'' (SFB F50).}
\address{Fakult{\"a}t f{\"u}r Mathematik, Universit{\"a}t Wien, Vienna, Austria}

\date{December 2016}


\email{robin.sulzgruber@univie.ac.at}
\begin{abstract}
The generating function of reverse plane partitions of a fixed shape factors into a product featuring the hook-lengths of this shape.
This result, which was first obtained by Stanley, can be explained bijectively using the Hillman--Grassl correspondence between reverse plane partitions and tableaux weighted by hook-lengths.
In this extended abstract an alternative bijection between the same families of objects is presented.
This construction is best perceived as a set of rules for building reverse plane partitions, viewed as arrangements of stacks of cubes, using bricks in the shape of rim-hooks.
\end{abstract}





\thispagestyle{empty}
\maketitle

Reverse plane partitions, like their relatives the plane partitions and semi-standard Young tableaux, appear naturally in the context of symmetric functions and the representation theory of permutation groups.
The generating function for reverse plane partitions was first obtained by Richard Stanley \cite[Prop.~18.3]{Stanley1971}.

\begin{thm}\label{Thm:rpp} The generating function for reverse plane partitions of shape $\lambda$ is given by
\begin{align*}
\sum_{\pi}q^{\abs{\pi}}
=\prod_{u\in\lambda}\frac{1}{1-q^{h(u)}},
\end{align*}
where $h(u)$ denotes the hook-length of the cell $u$.
\end{thm}

A bijective proof of this result was found by Hillman and Grassl~\cite{HilGra1976}.
Their algorithm has since been related to the Robinson--Schensted--Knuth correspondence and Sch{\"u}tzenberger's jeu de taquin~\cite{Gansner1981,Kadell1997} and has become an integral part of enumerative combinatorics.
In particular, Gansner~\cite[Thm.~3.2]{Gansner1981} used the Hillman--Grassl correspondence to obtain a refined generating function.

\begin{thm}\label{Thm:trace} The trace generating function for reverse plane partitions of shape $\lambda$ is given by
\begin{align*}
\sum_{\pi}\prod_{k\in\Z}q_k^{\tr_k(\pi)}
&=\prod_{u\in\lambda}\frac{1}{1-q^{H(u)}}
\quad\text{where}\quad
\tr_k(\pi)=\sum_{\substack{(i,j)\in\lambda\\j-i=k}}\pi(i,j)
\quad\text{and}\quad
q^{H(i,j)}=\prod_{k=j-\lambda_j'}^{\lambda_i-i}q_k.
\end{align*}
\end{thm}

More recently Morales, Panova and Pak~\cite{MorPanPak2016} used the Hillman--Grassl correspondence to obtain a (trace) generating function for reverse plane partitions of skew shape.

\smallskip
In this extended abstract we propose a different correspondence between reverse plane partitions and multi-sets of hooks, and thus provide an alternative bijective proof of Theorems~\ref{Thm:rpp} and~\ref{Thm:trace}.
Section~\ref{Section:blocks} contains an intuitive description of our algorithm,
while a formal definition is provided in Section~\ref{Section:insert} together with the needed background on partitions and reverse plane partitions.
The ``forward'' direction of our bijection attempts to insert rim-hooks into reverse plane partitions, using explicitly the shape of the rim-hook.
The ``inverse'' bijection, which is treated in Section~\ref{Section:factors}, extracts a rim-hook by decreasing a reverse plane partition along a path and is especially reminiscent of the method of Hillman and Grassl.

\smallskip
The relationship between the Hillman--Grassl correspondence and our new bijection is unclear at this time, but demands to be investigated.
Whether some of the deeper connections and applications of the Hillman--Grassl correspondence mentioned above are also reflected in this new correspondence is another interesting problem.

\section{Building with bricks}\label{Section:blocks}

This section aims to give an informal description of a simple set of rules how to build reverse plane partitions using bricks of rim-hook shape.
A precise mathematical formulation of the algorithm and all involved definitions are contained in Section~\ref{Section:insert}.

\smallskip
A partition is a left and top justified arrangement of cells.
A reverse plane partition is a left and top justified arrangement of numbers such that rows and columns are weakly increasing.
Alternatively we may represent reverse plane partitions by stacks of cubes.
Examples of all three objects are shown below.
Note that we have rotated the reverse plane partition to obtain the representation by cubes such that all stacks of cubes are visible.

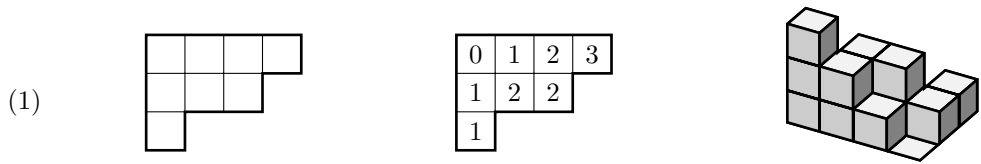
\begin{align}
\label{eq:0123-122-1}
\raisebox{-2em}{
\begin{tikzpicture}[scale=.6]
\begin{scope}
\draw
(0,1)--(1,1)
(0,2)--(3,2)
(1,1)--(1,3)
(2,1)--(2,3)
(3,2)--(3,3)
;
\draw[line width=1pt](0,0)--(1,0)--(1,1)--(3,1)--(3,2)--(4,2)--(4,3)--(0,3)--cycle;
\end{scope}
\begin{scope}[xshift=8cm]
\draw
(0,1)--(1,1)
(0,2)--(3,2)
(1,1)--(1,3)
(2,1)--(2,3)
(3,2)--(3,3)
;
\draw[line width=1pt](0,0)--(1,0)--(1,1)--(3,1)--(3,2)--(4,2)--(4,3)--(0,3)--cycle;
\draw[xshift=5mm,yshift=5mm]
(0,2)node{{$0$}}
(1,2)node{{$1$}}
(2,2)node{{$2$}}
(3,2)node{{$3$}}
(0,1)node{{$1$}}
(1,1)node{{$2$}}
(2,1)node{{$2$}}
(0,0)node{{$1$}};
\end{scope}
%
%
\begin{scope}[xshift=18cm,yshift=2cm,tdplot_main_coords]
\clip(3,0,0)--(3,4,0)--(0,4,0)--(0,4,1)--(0,3,1)--(1,3,1)--(1,3,2)--(1,2,2)--(1,1,2)--(2,1,2)--(2,1,3)--(2,0,3)--(3,0,3)--(3,0,0)--cycle;
\draw[line width=1pt,fill=black!5]
(2,3,0)--(3,3,0)--(3,4,0)--(2,4,0)--cycle
(2,2,1)--(3,2,1)--(3,3,1)--(2,3,1)--cycle
(2,1,2)--(3,1,2)--(3,2,2)--(2,2,2)--cycle
(2,0,3)--(3,0,3)--(3,1,3)--(2,1,3)--cycle
(1,3,1)--(2,3,1)--(2,4,1)--(1,4,1)--cycle
(1,2,2)--(2,2,2)--(2,3,2)--(1,3,2)--cycle
(1,1,2)--(2,1,2)--(2,2,2)--(1,2,2)--cycle
(0,3,1)--(1,3,1)--(1,4,1)--(0,4,1)--cycle
;
\draw[line width=1pt,fill=black!20]
(3,2,0)--(3,3,0)--(3,3,1)--(3,2,1)--cycle
(3,1,1)--(3,2,1)--(3,2,2)--(3,1,2)--cycle
(3,1,0)--(3,2,0)--(3,2,1)--(3,1,1)--cycle
(3,0,2)--(3,1,2)--(3,1,3)--(3,0,3)--cycle
(3,0,1)--(3,1,1)--(3,1,2)--(3,0,2)--cycle
(3,0,0)--(3,1,0)--(3,1,1)--(3,0,1)--cycle
(2,3,0)--(2,4,0)--(2,4,1)--(2,3,1)--cycle
(2,2,1)--(2,3,1)--(2,3,2)--(2,2,2)--cycle
;
\draw[line width=1pt,fill=black!50]
(2,3,0)--(2,3,1)--(3,3,1)--(3,3,0)--cycle
(2,2,1)--(2,2,2)--(3,2,2)--(3,2,1)--cycle
(2,1,2)--(2,1,3)--(3,1,3)--(3,1,2)--cycle
(1,4,0)--(1,4,1)--(2,4,1)--(2,4,0)--cycle
(1,3,1)--(1,3,2)--(2,3,2)--(2,3,1)--cycle
(0,4,0)--(0,4,1)--(1,4,1)--(1,4,0)--cycle
;
\draw[line width=2pt](3,0,0)--(3,4,0)--(0,4,0)--(0,4,1)--(0,3,1)--(1,3,1)--(1,3,2)--(1,2,2)--(1,1,2)--(2,1,2)--(2,1,3)--(2,0,3)--(3,0,3)--(3,0,0)--cycle;
\end{scope}
\end{tikzpicture}
}
\end{align}

Let $u$ be a cell of a partition $\lambda$.
The rim-hook of $u$ is the set of cells of $\lambda$ visited when following along the rim of $\lambda$ starting in the column of $u$ and ending in the row of $u$.
We view rim-hooks as stacks of cubes of height one.
The example below shows all eight rim-hooks of the partition $\lambda=(4,3,1)$.

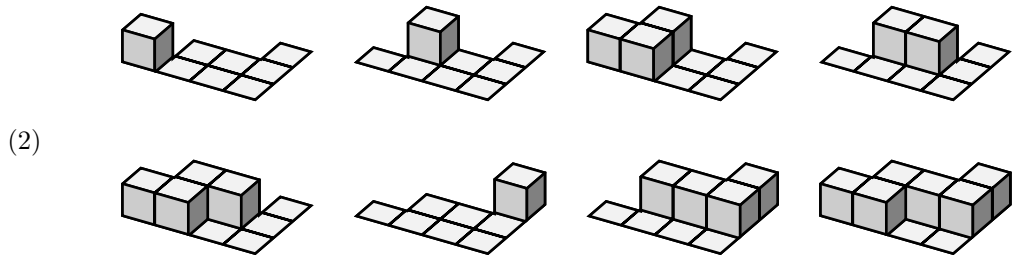
\begin{align}
\label{eq:rimhooks}
\raisebox{-4em}{
\begin{tikzpicture}[scale=.6,tdplot_main_coords]
\begin{scope}[xshift=0cm,yshift=0cm]
\clip(3,0,0)--(3,4,0)--(0,4,0)--(0,3,0)--(1,3,0)--(1,2,0)--(1,1,0)--(2,1,0)--(2,1,1)--(2,0,1)--(3,0,1)--(3,0,0)--cycle;
\draw[line width=1pt,fill=black!5]
(2,3,0)--(3,3,0)--(3,4,0)--(2,4,0)--cycle
(2,2,0)--(3,2,0)--(3,3,0)--(2,3,0)--cycle
(2,1,0)--(3,1,0)--(3,2,0)--(2,2,0)--cycle
(2,0,1)--(3,0,1)--(3,1,1)--(2,1,1)--cycle
(1,3,0)--(2,3,0)--(2,4,0)--(1,4,0)--cycle
(1,2,0)--(2,2,0)--(2,3,0)--(1,3,0)--cycle
(1,1,0)--(2,1,0)--(2,2,0)--(1,2,0)--cycle
(0,3,0)--(1,3,0)--(1,4,0)--(0,4,0)--cycle
;
\draw[line width=1pt,fill=black!20]
(3,0,0)--(3,1,0)--(3,1,1)--(3,0,1)--cycle
;
\draw[line width=1pt,fill=black!50]
(2,1,0)--(2,1,1)--(3,1,1)--(3,1,0)--cycle
;
\draw[line width=2pt](3,0,0)--(3,4,0)--(0,4,0)--(0,3,0)--(1,3,0)--(1,2,0)--(1,1,0)--(2,1,0)--(2,1,1)--(2,0,1)--(3,0,1)--(3,0,0)--cycle;
\end{scope}
\begin{scope}[xshift=6cm,yshift=0cm]
\clip(3,0,0)--(3,4,0)--(0,4,0)--(0,3,0)--(1,3,0)--(1,2,0)--(1,2,1)--(1,1,1)--(2,1,1)--(2,1,0)--(2,0,0)--(3,0,0)--cycle;
\draw[line width=1pt,fill=black!5]
(2,3,0)--(3,3,0)--(3,4,0)--(2,4,0)--cycle
(2,2,0)--(3,2,0)--(3,3,0)--(2,3,0)--cycle
(2,1,0)--(3,1,0)--(3,2,0)--(2,2,0)--cycle
(2,0,0)--(3,0,0)--(3,1,0)--(2,1,0)--cycle
(1,3,0)--(2,3,0)--(2,4,0)--(1,4,0)--cycle
(1,2,0)--(2,2,0)--(2,3,0)--(1,3,0)--cycle
(1,1,1)--(2,1,1)--(2,2,1)--(1,2,1)--cycle
(0,3,0)--(1,3,0)--(1,4,0)--(0,4,0)--cycle
;
\draw[line width=1pt,fill=black!20]
(2,1,0)--(2,2,0)--(2,2,1)--(2,1,1)--cycle
;
\draw[line width=1pt,fill=black!50]
(1,2,0)--(1,2,1)--(2,2,1)--(2,2,0)--cycle
;
\draw[line width=2pt](3,0,0)--(3,4,0)--(0,4,0)--(0,3,0)--(1,3,0)--(1,2,0)--(1,2,1)--(1,1,1)--(2,1,1)--(2,1,0)--(2,0,0)--(3,0,0)--cycle;
\end{scope}
\begin{scope}[xshift=12cm,yshift=0cm]
\clip(3,0,0)--(3,4,0)--(0,4,0)--(0,3,0)--(1,3,0)--(1,2,0)--(1,2,1)--(1,1,1)--(2,1,1)--(2,0,1)--(3,0,1)--(3,0,0)--cycle;
\draw[line width=1pt,fill=black!5]
(2,3,0)--(3,3,0)--(3,4,0)--(2,4,0)--cycle
(2,2,0)--(3,2,0)--(3,3,0)--(2,3,0)--cycle
(2,1,1)--(3,1,1)--(3,2,1)--(2,2,1)--cycle
(2,0,1)--(3,0,1)--(3,1,1)--(2,1,1)--cycle
(1,3,0)--(2,3,0)--(2,4,0)--(1,4,0)--cycle
(1,2,0)--(2,2,0)--(2,3,0)--(1,3,0)--cycle
(1,1,1)--(2,1,1)--(2,2,1)--(1,2,1)--cycle
(0,3,0)--(1,3,0)--(1,4,0)--(0,4,0)--cycle
;
\draw[line width=1pt,fill=black!20]
(3,1,0)--(3,2,0)--(3,2,1)--(3,1,1)--cycle
(3,0,0)--(3,1,0)--(3,1,1)--(3,0,1)--cycle
;
\draw[line width=1pt,fill=black!50]
(2,2,0)--(2,2,1)--(3,2,1)--(3,2,0)--cycle
(1,2,0)--(1,2,1)--(2,2,1)--(2,2,0)--cycle
;
\draw[line width=2pt](3,0,0)--(3,4,0)--(0,4,0)--(0,3,0)--(1,3,0)--(1,2,0)--(1,2,1)--(1,1,1)--(2,1,1)--(2,0,1)--(3,0,1)--(3,0,0)--cycle;
\end{scope}
\begin{scope}[xshift=18cm,yshift=0cm]
\clip(3,0,0)--(3,4,0)--(0,4,0)--(0,3,0)--(1,3,0)--(1,3,1)--(1,2,1)--(1,1,1)--(2,1,1)--(2,1,0)--(2,0,0)--(3,0,0)--cycle;
\draw[line width=1pt,fill=black!5]
(2,3,0)--(3,3,0)--(3,4,0)--(2,4,0)--cycle
(2,2,0)--(3,2,0)--(3,3,0)--(2,3,0)--cycle
(2,1,0)--(3,1,0)--(3,2,0)--(2,2,0)--cycle
(2,0,0)--(3,0,0)--(3,1,0)--(2,1,0)--cycle
(1,3,0)--(2,3,0)--(2,4,0)--(1,4,0)--cycle
(1,2,1)--(2,2,1)--(2,3,1)--(1,3,1)--cycle
(1,1,1)--(2,1,1)--(2,2,1)--(1,2,1)--cycle
(0,3,0)--(1,3,0)--(1,4,0)--(0,4,0)--cycle
;
\draw[line width=1pt,fill=black!20]
(2,2,0)--(2,3,0)--(2,3,1)--(2,2,1)--cycle
(2,1,0)--(2,2,0)--(2,2,1)--(2,1,1)--cycle
;
\draw[line width=1pt,fill=black!50]
(1,3,0)--(1,3,1)--(2,3,1)--(2,3,0)--cycle
;
\draw[line width=2pt](3,0,0)--(3,4,0)--(0,4,0)--(0,3,0)--(1,3,0)--(1,3,1)--(1,2,1)--(1,1,1)--(2,1,1)--(2,1,0)--(2,0,0)--(3,0,0)--cycle;
\end{scope}
\begin{scope}[xshift=0cm,yshift=-4cm]
\clip(3,0,0)--(3,4,0)--(0,4,0)--(0,3,0)--(1,3,0)--(1,3,1)--(1,2,1)--(1,1,1)--(2,1,1)--(2,0,1)--(3,0,1)--(3,0,0)--cycle;
\draw[line width=1pt,fill=black!5]
(2,3,0)--(3,3,0)--(3,4,0)--(2,4,0)--cycle
(2,2,0)--(3,2,0)--(3,3,0)--(2,3,0)--cycle
(2,1,1)--(3,1,1)--(3,2,1)--(2,2,1)--cycle
(2,0,1)--(3,0,1)--(3,1,1)--(2,1,1)--cycle
(1,3,0)--(2,3,0)--(2,4,0)--(1,4,0)--cycle
(1,2,1)--(2,2,1)--(2,3,1)--(1,3,1)--cycle
(1,1,1)--(2,1,1)--(2,2,1)--(1,2,1)--cycle
(0,3,0)--(1,3,0)--(1,4,0)--(0,4,0)--cycle
;
\draw[line width=1pt,fill=black!20]
(3,1,0)--(3,2,0)--(3,2,1)--(3,1,1)--cycle
(3,0,0)--(3,1,0)--(3,1,1)--(3,0,1)--cycle
(2,2,0)--(2,3,0)--(2,3,1)--(2,2,1)--cycle
;
\draw[line width=1pt,fill=black!50]
(2,2,0)--(2,2,1)--(3,2,1)--(3,2,0)--cycle
(1,3,0)--(1,3,1)--(2,3,1)--(2,3,0)--cycle
;
\draw[line width=2pt](3,0,0)--(3,4,0)--(0,4,0)--(0,3,0)--(1,3,0)--(1,3,1)--(1,2,1)--(1,1,1)--(2,1,1)--(2,0,1)--(3,0,1)--(3,0,0)--cycle;
\end{scope}
\begin{scope}[xshift=6cm,yshift=-4cm]
\clip(3,0,0)--(3,4,0)--(0,4,0)--(0,4,1)--(0,3,1)--(1,3,1)--(1,3,0)--(1,2,0)--(1,1,0)--(2,1,0)--(2,0,0)--(3,0,0)--cycle;
\draw[line width=1pt,fill=black!5]
(2,3,0)--(3,3,0)--(3,4,0)--(2,4,0)--cycle
(2,2,0)--(3,2,0)--(3,3,0)--(2,3,0)--cycle
(2,1,0)--(3,1,0)--(3,2,0)--(2,2,0)--cycle
(2,0,0)--(3,0,0)--(3,1,0)--(2,1,0)--cycle
(1,3,0)--(2,3,0)--(2,4,0)--(1,4,0)--cycle
(1,2,0)--(2,2,0)--(2,3,0)--(1,3,0)--cycle
(1,1,0)--(2,1,0)--(2,2,0)--(1,2,0)--cycle
(0,3,1)--(1,3,1)--(1,4,1)--(0,4,1)--cycle
;
\draw[line width=1pt,fill=black!20]
(1,3,0)--(1,4,0)--(1,4,1)--(1,3,1)--cycle
;
\draw[line width=1pt,fill=black!50]
(0,4,0)--(0,4,1)--(1,4,1)--(1,4,0)--cycle
;
\draw[line width=2pt](3,0,0)--(3,4,0)--(0,4,0)--(0,4,1)--(0,3,1)--(1,3,1)--(1,3,0)--(1,2,0)--(1,1,0)--(2,1,0)--(2,0,0)--(3,0,0)--cycle;
\end{scope}
\begin{scope}[xshift=12cm,yshift=-4cm]
\clip(3,0,0)--(3,4,0)--(0,4,0)--(0,4,1)--(0,3,1)--(1,3,1)--(1,2,1)--(1,1,1)--(2,1,1)--(2,1,0)--(2,0,0)--(3,0,0)--cycle;
\draw[line width=1pt,fill=black!5]
(2,3,0)--(3,3,0)--(3,4,0)--(2,4,0)--cycle
(2,2,0)--(3,2,0)--(3,3,0)--(2,3,0)--cycle
(2,1,0)--(3,1,0)--(3,2,0)--(2,2,0)--cycle
(2,0,0)--(3,0,0)--(3,1,0)--(2,1,0)--cycle
(1,3,1)--(2,3,1)--(2,4,1)--(1,4,1)--cycle
(1,2,1)--(2,2,1)--(2,3,1)--(1,3,1)--cycle
(1,1,1)--(2,1,1)--(2,2,1)--(1,2,1)--cycle
(0,3,1)--(1,3,1)--(1,4,1)--(0,4,1)--cycle
;
\draw[line width=1pt,fill=black!20]
(2,3,0)--(2,4,0)--(2,4,1)--(2,3,1)--cycle
(2,2,0)--(2,3,0)--(2,3,1)--(2,2,1)--cycle
(2,1,0)--(2,2,0)--(2,2,1)--(2,1,1)--cycle
;
\draw[line width=1pt,fill=black!50]
(1,4,0)--(1,4,1)--(2,4,1)--(2,4,0)--cycle
(0,4,0)--(0,4,1)--(1,4,1)--(1,4,0)--cycle
;
\draw[line width=2pt](3,0,0)--(3,4,0)--(0,4,0)--(0,4,1)--(0,3,1)--(1,3,1)--(1,2,1)--(1,1,1)--(2,1,1)--(2,1,0)--(2,0,0)--(3,0,0)--cycle;
\end{scope}
\begin{scope}[xshift=18cm,yshift=-4cm]
\clip(3,0,0)--(3,4,0)--(0,4,0)--(0,4,1)--(0,3,1)--(1,3,1)--(1,2,1)--(1,1,1)--(2,1,1)--(2,0,1)--(3,0,1)--(3,0,0)--cycle;
\draw[line width=1pt,fill=black!5]
(2,3,0)--(3,3,0)--(3,4,0)--(2,4,0)--cycle
(2,2,0)--(3,2,0)--(3,3,0)--(2,3,0)--cycle
(2,1,1)--(3,1,1)--(3,2,1)--(2,2,1)--cycle
(2,0,1)--(3,0,1)--(3,1,1)--(2,1,1)--cycle
(1,3,1)--(2,3,1)--(2,4,1)--(1,4,1)--cycle
(1,2,1)--(2,2,1)--(2,3,1)--(1,3,1)--cycle
(1,1,1)--(2,1,1)--(2,2,1)--(1,2,1)--cycle
(0,3,1)--(1,3,1)--(1,4,1)--(0,4,1)--cycle
;
\draw[line width=1pt,fill=black!20]
(3,1,0)--(3,2,0)--(3,2,1)--(3,1,1)--cycle
(3,0,0)--(3,1,0)--(3,1,1)--(3,0,1)--cycle
(2,3,0)--(2,4,0)--(2,4,1)--(2,3,1)--cycle
(2,2,0)--(2,3,0)--(2,3,1)--(2,2,1)--cycle
;
\draw[line width=1pt,fill=black!50]
(2,2,0)--(2,2,1)--(3,2,1)--(3,2,0)--cycle
(1,4,0)--(1,4,1)--(2,4,1)--(2,4,0)--cycle
(0,4,0)--(0,4,1)--(1,4,1)--(1,4,0)--cycle
;
\draw[line width=2pt](3,0,0)--(3,4,0)--(0,4,0)--(0,4,1)--(0,3,1)--(1,3,1)--(1,2,1)--(1,1,1)--(2,1,1)--(2,0,1)--(3,0,1)--(3,0,0)--cycle;
\end{scope}
\end{tikzpicture}
}
\end{align}

A part of a rim-hook is a maximal set of cubes contained in the same row of the underlying partition.
For example, the following rim-hook has three parts of sizes two, three and one.

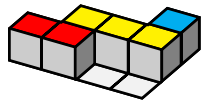
\begin{align*}
\begin{tikzpicture}[scale=.6,tdplot_main_coords]
\begin{scope}[xshift=0cm,yshift=0cm]
\clip(3,0,0)--(3,4,0)--(0,4,0)--(0,4,1)--(0,3,1)--(1,3,1)--(1,2,1)--(1,1,1)--(2,1,1)--(2,0,1)--(3,0,1)--(3,0,0)--cycle;
\draw[line width=1pt,fill=black!5]
(2,3,0)--(3,3,0)--(3,4,0)--(2,4,0)--cycle
(2,2,0)--(3,2,0)--(3,3,0)--(2,3,0)--cycle
;
\draw[line width=1pt,fill=red]
(2,0,1)--(3,0,1)--(3,1,1)--(2,1,1)--cycle
(2,1,1)--(3,1,1)--(3,2,1)--(2,2,1)--cycle
;
\draw[line width=1pt,fill=yellow]
(1,3,1)--(2,3,1)--(2,4,1)--(1,4,1)--cycle
(1,2,1)--(2,2,1)--(2,3,1)--(1,3,1)--cycle
(1,1,1)--(2,1,1)--(2,2,1)--(1,2,1)--cycle
;
\draw[line width=1pt,fill=cyan]
(0,3,1)--(1,3,1)--(1,4,1)--(0,4,1)--cycle
;
\draw[line width=1pt,fill=black!20]
(3,1,0)--(3,2,0)--(3,2,1)--(3,1,1)--cycle
(3,0,0)--(3,1,0)--(3,1,1)--(3,0,1)--cycle
(2,3,0)--(2,4,0)--(2,4,1)--(2,3,1)--cycle
(2,2,0)--(2,3,0)--(2,3,1)--(2,2,1)--cycle
;
\draw[line width=1pt,fill=black!50]
(2,2,0)--(2,2,1)--(3,2,1)--(3,2,0)--cycle
(1,4,0)--(1,4,1)--(2,4,1)--(2,4,0)--cycle
(0,4,0)--(0,4,1)--(1,4,1)--(1,4,0)--cycle
;
\draw[line width=2pt](3,0,0)--(3,4,0)--(0,4,0)--(0,4,1)--(0,3,1)--(1,3,1)--(1,2,1)--(1,1,1)--(2,1,1)--(2,0,1)--(3,0,1)--(3,0,0)--cycle;
\end{scope}
\end{tikzpicture}
\end{align*}

Suppose you are given a reverse plane partition $\pi$ of shape $\lambda$ and a rim-hook $h$ of $\lambda$.
To insert $h$ into $\pi$ first try placing $h$ on top of $\pi$ such that the ends of $h$ align with the appropriate cells of the shape of $\pi$.
If the result is a reverse plane partition, you are done!

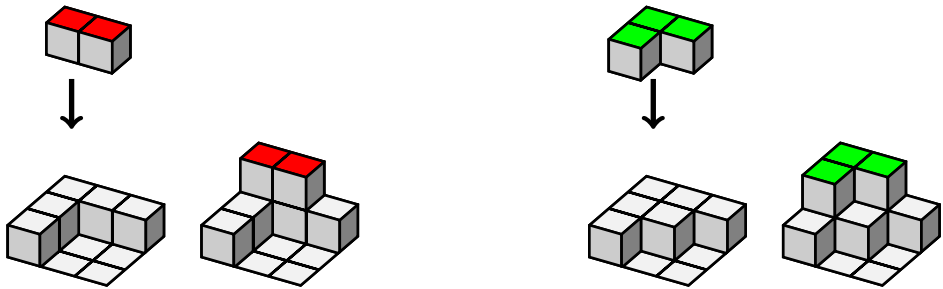
\begin{align*}
\begin{tikzpicture}[scale=.6,tdplot_main_coords]
\begin{scope}
\begin{scope}[xshift=0cm,yshift=-2cm]
\draw[line width=2pt,->](.5,.5,1.5)--(.5,.5,0);
\end{scope}
\begin{scope}[xshift=0cm,yshift=0cm]
\clip(1,0,0)--(1,2,0)--(0,2,0)--(0,2,1)--(0,1,1)--(0,1,1)--(0,0,1)--(1,0,1)--cycle;
\draw[line width=1pt,fill=red]
(0,0,1)--(1,0,1)--(1,1,1)--(0,1,1)--cycle
(0,1,1)--(1,1,1)--(1,2,1)--(0,2,1)--cycle
;
\draw[line width=1pt,fill=black!20]
(1,0,1)--(1,0,0)--(1,1,0)--(1,1,1)--cycle
(1,1,1)--(1,1,0)--(1,2,0)--(1,2,1)--cycle
;
\draw[line width=1pt,fill=black!50]
(0,2,1)--(1,2,1)--(1,2,0)--(0,2,0)--cycle
;
\draw[line width=2pt](1,0,0)--(1,2,0)--(0,2,0)--(0,2,1)--(0,1,1)--(0,1,1)--(0,0,1)--(1,0,1)--cycle;
\end{scope}
\begin{scope}[yshift=-44mm]
\begin{scope}
\clip(3,0,0)--(3,3,0)--(0,3,0)--(0,3,1)--(0,1,1)--(0,1,1)--(0,0,1)--(3,0,1)--cycle;
\draw[line width=1pt,fill=black!5]
(0,0,1)--(1,0,1)--(1,1,1)--(0,1,1)--cycle
(0,1,1)--(1,1,1)--(1,2,1)--(0,2,1)--cycle
(0,2,1)--(1,2,1)--(1,3,1)--(0,3,1)--cycle
(1,0,1)--(2,0,1)--(2,1,1)--(1,1,1)--cycle
(1,1,0)--(2,1,0)--(2,2,0)--(1,2,0)--cycle
(1,2,0)--(2,2,0)--(2,3,0)--(1,3,0)--cycle
(2,0,1)--(3,0,1)--(3,1,1)--(2,1,1)--cycle
(2,1,0)--(3,1,0)--(3,2,0)--(2,2,0)--cycle
(2,2,0)--(3,2,0)--(3,3,0)--(2,3,0)--cycle
;
\draw[line width=1pt,fill=black!20]
(1,1,1)--(1,1,0)--(1,2,0)--(1,2,1)--cycle
(1,2,1)--(1,2,0)--(1,3,0)--(1,3,1)--cycle
(3,0,1)--(3,0,0)--(3,1,0)--(3,1,1)--cycle
;
\draw[line width=1pt,fill=black!50]
(0,3,1)--(1,3,1)--(1,3,0)--(0,3,0)--cycle
(1,1,1)--(2,1,1)--(2,1,0)--(1,1,0)--cycle
(2,1,1)--(3,1,1)--(3,1,0)--(2,1,0)--cycle
;
\draw[line width=2pt](3,0,0)--(3,3,0)--(0,3,0)--(0,3,1)--(0,1,1)--(0,1,1)--(0,0,1)--(3,0,1)--cycle;
\end{scope}
\begin{scope}[xshift=5cm]
\clip(3,0,0)--(3,3,0)--(0,3,0)--(0,3,1)--(0,2,1)--(0,2,2)--(0,1,2)--(0,1,2)--(0,0,2)--(1,0,2)--(1,0,1)--(3,0,1)--cycle;
\draw[line width=1pt,fill=black!5]
(0,2,1)--(1,2,1)--(1,3,1)--(0,3,1)--cycle
(1,0,1)--(2,0,1)--(2,1,1)--(1,1,1)--cycle
(1,1,0)--(2,1,0)--(2,2,0)--(1,2,0)--cycle
(1,2,0)--(2,2,0)--(2,3,0)--(1,3,0)--cycle
(2,0,1)--(3,0,1)--(3,1,1)--(2,1,1)--cycle
(2,1,0)--(3,1,0)--(3,2,0)--(2,2,0)--cycle
(2,2,0)--(3,2,0)--(3,3,0)--(2,3,0)--cycle
;
\draw[line width=1pt,fill=red]
(0,0,2)--(1,0,2)--(1,1,2)--(0,1,2)--cycle
(0,1,2)--(1,1,2)--(1,2,2)--(0,2,2)--cycle
;
\draw[line width=1pt,fill=black!20]
(1,0,2)--(1,0,1)--(1,1,1)--(1,1,2)--cycle
(1,1,2)--(1,1,1)--(1,2,1)--(1,2,2)--cycle
(1,1,1)--(1,1,0)--(1,2,0)--(1,2,1)--cycle
(1,2,1)--(1,2,0)--(1,3,0)--(1,3,1)--cycle
(3,0,1)--(3,0,0)--(3,1,0)--(3,1,1)--cycle
;
\draw[line width=1pt,fill=black!50]
(0,2,2)--(1,2,2)--(1,2,1)--(0,2,1)--cycle
(0,3,1)--(1,3,1)--(1,3,0)--(0,3,0)--cycle
(1,1,1)--(2,1,1)--(2,1,0)--(1,1,0)--cycle
(2,1,1)--(3,1,1)--(3,1,0)--(2,1,0)--cycle
;
\draw[line width=2pt](3,0,0)--(3,3,0)--(0,3,0)--(0,3,1)--(0,2,1)--(0,2,2)--(0,1,2)--(0,1,2)--(0,0,2)--(1,0,2)--(1,0,1)--(3,0,1)--cycle;
\end{scope}
\end{scope}
\end{scope}
\begin{scope}[xshift=15cm]
\begin{scope}[xshift=0cm,yshift=-2cm]
\draw[line width=2pt,->](.5,.5,1.5)--(.5,.5,0);
\end{scope}
\begin{scope}[xshift=0cm,yshift=0cm]
\clip(2,0,0)--(2,1,0)--(1,1,0)--(1,2,0)--(0,2,0)--(0,2,1)--(0,1,1)--(0,1,1)--(0,0,1)--(2,0,1)--cycle;
\draw[line width=1pt,fill=green]
(0,0,1)--(1,0,1)--(1,1,1)--(0,1,1)--cycle
(0,1,1)--(1,1,1)--(1,2,1)--(0,2,1)--cycle
(1,0,1)--(2,0,1)--(2,1,1)--(1,1,1)--cycle
;
\draw[line width=1pt,fill=black!20]
(1,1,1)--(1,1,0)--(1,2,0)--(1,2,1)--cycle
(2,0,1)--(2,0,0)--(2,1,0)--(2,1,1)--cycle
;
\draw[line width=1pt,fill=black!50]
(0,2,1)--(1,2,1)--(1,2,0)--(0,2,0)--cycle
(1,1,1)--(2,1,1)--(2,1,0)--(1,1,0)--cycle
;
\draw[line width=2pt](2,0,0)--(2,1,0)--(1,1,0)--(1,2,0)--(0,2,0)--(0,2,1)--(0,1,1)--(0,1,1)--(0,0,1)--(2,0,1)--cycle;
\end{scope}
\begin{scope}[yshift=-44mm]
\begin{scope}[xshift=0cm,yshift=0cm]
\clip(3,0,0)--(3,3,0)--(0,3,0)--(0,3,1)--(0,1,1)--(0,1,1)--(0,0,1)--(3,0,1)--cycle;
\draw[line width=1pt,fill=black!5]
(0,0,1)--(1,0,1)--(1,1,1)--(0,1,1)--cycle
(0,1,1)--(1,1,1)--(1,2,1)--(0,2,1)--cycle
(0,2,1)--(1,2,1)--(1,3,1)--(0,3,1)--cycle
(1,2,0)--(2,2,0)--(2,3,0)--(1,3,0)--cycle
(2,1,0)--(3,1,0)--(3,2,0)--(2,2,0)--cycle
(2,2,0)--(3,2,0)--(3,3,0)--(2,3,0)--cycle
(1,0,1)--(2,0,1)--(2,1,1)--(1,1,1)--cycle
(2,0,1)--(3,0,1)--(3,1,1)--(2,1,1)--cycle
(1,1,1)--(2,1,1)--(2,2,1)--(1,2,1)--cycle
;
\draw[line width=1pt,fill=black!20]
(1,2,1)--(1,2,0)--(1,3,0)--(1,3,1)--cycle
(2,1,1)--(2,1,0)--(2,2,0)--(2,2,1)--cycle
(3,0,1)--(3,0,0)--(3,1,0)--(3,1,1)--cycle
;
\draw[line width=1pt,fill=black!50]
(0,3,1)--(1,3,1)--(1,3,0)--(0,3,0)--cycle
(1,2,1)--(2,2,1)--(2,2,0)--(1,2,0)--cycle
(2,1,1)--(3,1,1)--(3,1,0)--(2,1,0)--cycle
;
\draw[line width=2pt](3,0,0)--(3,3,0)--(0,3,0)--(0,3,1)--(0,1,1)--(0,1,1)--(0,0,1)--(3,0,1)--cycle;
\end{scope}
\begin{scope}[xshift=5cm,yshift=0cm]
\clip(3,0,0)--(3,3,0)--(0,3,0)--(0,3,1)--(0,2,1)--(0,2,2)--(0,1,2)--(0,1,2)--(0,0,2)--(2,0,2)--(2,0,1)--(3,0,1)--cycle;
\draw[line width=1pt,fill=black!5]
(0,2,1)--(1,2,1)--(1,3,1)--(0,3,1)--cycle
(1,1,1)--(2,1,1)--(2,2,1)--(1,2,1)--cycle
(1,2,0)--(2,2,0)--(2,3,0)--(1,3,0)--cycle
(2,0,1)--(3,0,1)--(3,1,1)--(2,1,1)--cycle
(2,1,0)--(3,1,0)--(3,2,0)--(2,2,0)--cycle
(2,2,0)--(3,2,0)--(3,3,0)--(2,3,0)--cycle
;
\draw[line width=1pt,fill=green]
(0,0,2)--(1,0,2)--(1,1,2)--(0,1,2)--cycle
(0,1,2)--(1,1,2)--(1,2,2)--(0,2,2)--cycle
(1,0,2)--(2,0,2)--(2,1,2)--(1,1,2)--cycle
;
\draw[line width=1pt,fill=black!20]
(1,1,2)--(1,1,1)--(1,2,1)--(1,2,2)--cycle
(1,2,1)--(1,2,0)--(1,3,0)--(1,3,1)--cycle
(2,0,2)--(2,0,1)--(2,1,1)--(2,1,2)--cycle
(2,1,1)--(2,1,0)--(2,2,0)--(2,2,1)--cycle
(3,0,1)--(3,0,0)--(3,1,0)--(3,1,1)--cycle
;
\draw[line width=1pt,fill=black!50]
(0,2,2)--(1,2,2)--(1,2,1)--(0,2,1)--cycle
(0,3,1)--(1,3,1)--(1,3,0)--(0,3,0)--cycle
(1,1,2)--(2,1,2)--(2,1,1)--(1,1,1)--cycle
(1,2,1)--(2,2,1)--(2,2,0)--(1,2,0)--cycle
(2,1,1)--(3,1,1)--(3,1,0)--(2,1,0)--cycle
;
\draw[line width=2pt](3,0,0)--(3,3,0)--(0,3,0)--(0,3,1)--(0,2,1)--(0,2,2)--(0,1,2)--(0,1,2)--(0,0,2)--(2,0,2)--(2,0,1)--(3,0,1)--cycle;
\end{scope}
\end{scope}
\end{scope}
\end{tikzpicture}
\end{align*}

If you are not so lucky the rim-hook might not fit as easily, that is, the resulting arrangement of cubes might have holes.
In this case try the following:
Cut off the maximum number of parts that can be inserted in a single layer of the reverse plane partition and insert this initial segment of the rim-hook.
Then shift the remainder of $h$ diagonally by one cell and try to insert it in the new position.

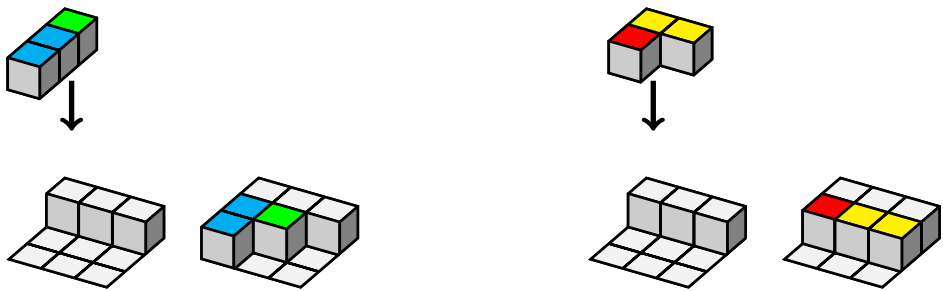
\begin{align*}
\begin{tikzpicture}[scale=.6,tdplot_main_coords]
\begin{scope}
\begin{scope}[xshift=0cm,yshift=-2cm]
\draw[line width=2pt,->](.5,.5,1.5)--(.5,.5,0);
\end{scope}
\begin{scope}[xshift=0cm,yshift=0cm]
\clip(3,0,0)--(3,1,0)--(0,1,0)--(0,1,1)--(0,0,1)--(3,0,1)--cycle;
\draw[line width=1pt,fill=green]
(0,0,1)--(1,0,1)--(1,1,1)--(0,1,1)--cycle
;
\draw[line width=1pt,fill=cyan]
(1,0,1)--(2,0,1)--(2,1,1)--(1,1,1)--cycle
(2,0,1)--(3,0,1)--(3,1,1)--(2,1,1)--cycle
;
\draw[line width=1pt,fill=black!20]
(3,0,1)--(3,0,0)--(3,1,0)--(3,1,1)--cycle
;
\draw[line width=1pt,fill=black!50]
(0,1,1)--(1,1,1)--(1,1,0)--(0,1,0)--cycle
(1,1,1)--(2,1,1)--(2,1,0)--(1,1,0)--cycle
(2,1,1)--(3,1,1)--(3,1,0)--(2,1,0)--cycle
;
\draw[line width=2pt](3,0,0)--(3,1,0)--(0,1,0)--(0,1,1)--(0,0,1)--(3,0,1)--cycle;
\end{scope}
\begin{scope}[yshift=-44mm]
\begin{scope}[xshift=0cm,yshift=0cm]
\clip(3,0,0)--(3,3,0)--(0,3,0)--(0,3,1)--(0,1,1)--(0,1,1)--(0,0,1)--(1,0,1)--(1,0,0)--(3,0,0)--cycle;
\draw[line width=1pt,fill=black!5]
(0,0,1)--(1,0,1)--(1,1,1)--(0,1,1)--cycle
(0,1,1)--(1,1,1)--(1,2,1)--(0,2,1)--cycle
(0,2,1)--(1,2,1)--(1,3,1)--(0,3,1)--cycle
(1,0,0)--(2,0,0)--(2,1,0)--(1,1,0)--cycle
(1,1,0)--(2,1,0)--(2,2,0)--(1,2,0)--cycle
(1,2,0)--(2,2,0)--(2,3,0)--(1,3,0)--cycle
(2,0,0)--(3,0,0)--(3,1,0)--(2,1,0)--cycle
(2,1,0)--(3,1,0)--(3,2,0)--(2,2,0)--cycle
(2,2,0)--(3,2,0)--(3,3,0)--(2,3,0)--cycle
;
\draw[line width=1pt,fill=black!20]
(1,0,1)--(1,0,0)--(1,1,0)--(1,1,1)--cycle
(1,1,1)--(1,1,0)--(1,2,0)--(1,2,1)--cycle
(1,2,1)--(1,2,0)--(1,3,0)--(1,3,1)--cycle
;
\draw[line width=1pt,fill=black!50]
(0,3,1)--(1,3,1)--(1,3,0)--(0,3,0)--cycle
;
\draw[line width=2pt](3,0,0)--(3,3,0)--(0,3,0)--(0,3,1)--(0,1,1)--(0,1,1)--(0,0,1)--(1,0,1)--(1,0,0)--(3,0,0)--cycle;
\end{scope}
\begin{scope}[xshift=5cm,yshift=0cm]
\clip(3,0,0)--(3,3,0)--(0,3,0)--(0,3,1)--(0,1,1)--(0,1,1)--(0,0,1)--(3,0,1)--cycle;
\draw[line width=1pt,fill=black!5]
(0,0,1)--(1,0,1)--(1,1,1)--(0,1,1)--cycle
(0,1,1)--(1,1,1)--(1,2,1)--(0,2,1)--cycle
(0,2,1)--(1,2,1)--(1,3,1)--(0,3,1)--cycle
(1,2,0)--(2,2,0)--(2,3,0)--(1,3,0)--cycle
(2,1,0)--(3,1,0)--(3,2,0)--(2,2,0)--cycle
(2,2,0)--(3,2,0)--(3,3,0)--(2,3,0)--cycle
;
\draw[line width=1pt,fill=cyan]
(1,0,1)--(2,0,1)--(2,1,1)--(1,1,1)--cycle
(2,0,1)--(3,0,1)--(3,1,1)--(2,1,1)--cycle
;
\draw[line width=1pt,fill=green]
(1,1,1)--(2,1,1)--(2,2,1)--(1,2,1)--cycle
;
\draw[line width=1pt,fill=black!20]
(1,2,1)--(1,2,0)--(1,3,0)--(1,3,1)--cycle
(2,1,1)--(2,1,0)--(2,2,0)--(2,2,1)--cycle
(3,0,1)--(3,0,0)--(3,1,0)--(3,1,1)--cycle
;
\draw[line width=1pt,fill=black!50]
(0,3,1)--(1,3,1)--(1,3,0)--(0,3,0)--cycle
(1,2,1)--(2,2,1)--(2,2,0)--(1,2,0)--cycle
(2,1,1)--(3,1,1)--(3,1,0)--(2,1,0)--cycle
;
\draw[line width=2pt](3,0,0)--(3,3,0)--(0,3,0)--(0,3,1)--(0,1,1)--(0,1,1)--(0,0,1)--(3,0,1)--cycle;
\end{scope}
\end{scope}
\end{scope}
\begin{scope}[xshift=15cm]
\begin{scope}[xshift=0cm,yshift=-2cm]
\draw[line width=2pt,->](.5,.5,1.5)--(.5,.5,0);
\end{scope}
\begin{scope}[xshift=0cm,yshift=0cm]
\clip(2,0,0)--(2,1,0)--(1,1,0)--(1,2,0)--(0,2,0)--(0,2,1)--(0,1,1)--(0,1,1)--(0,0,1)--(2,0,1)--cycle;
\draw[line width=1pt,fill=red]
(1,0,1)--(2,0,1)--(2,1,1)--(1,1,1)--cycle
;
\draw[line width=1pt,fill=yellow]
(0,0,1)--(1,0,1)--(1,1,1)--(0,1,1)--cycle
(0,1,1)--(1,1,1)--(1,2,1)--(0,2,1)--cycle
;
\draw[line width=1pt,fill=black!20]
(1,1,1)--(1,1,0)--(1,2,0)--(1,2,1)--cycle
(2,0,1)--(2,0,0)--(2,1,0)--(2,1,1)--cycle
;
\draw[line width=1pt,fill=black!50]
(0,2,1)--(1,2,1)--(1,2,0)--(0,2,0)--cycle
(1,1,1)--(2,1,1)--(2,1,0)--(1,1,0)--cycle
;
\draw[line width=2pt](2,0,0)--(2,1,0)--(1,1,0)--(1,2,0)--(0,2,0)--(0,2,1)--(0,1,1)--(0,1,1)--(0,0,1)--(2,0,1)--cycle;
\end{scope}
\begin{scope}[yshift=-44mm]
\begin{scope}[xshift=0cm,yshift=0cm]
\clip(3,0,0)--(3,3,0)--(0,3,0)--(0,3,1)--(0,1,1)--(0,1,1)--(0,0,1)--(1,0,1)--(1,0,0)--(3,0,0)--cycle;
\draw[line width=1pt,fill=black!5]
(0,0,1)--(1,0,1)--(1,1,1)--(0,1,1)--cycle
(0,1,1)--(1,1,1)--(1,2,1)--(0,2,1)--cycle
(0,2,1)--(1,2,1)--(1,3,1)--(0,3,1)--cycle
(1,0,0)--(2,0,0)--(2,1,0)--(1,1,0)--cycle
(1,1,0)--(2,1,0)--(2,2,0)--(1,2,0)--cycle
(1,2,0)--(2,2,0)--(2,3,0)--(1,3,0)--cycle
(2,0,0)--(3,0,0)--(3,1,0)--(2,1,0)--cycle
(2,1,0)--(3,1,0)--(3,2,0)--(2,2,0)--cycle
(2,2,0)--(3,2,0)--(3,3,0)--(2,3,0)--cycle
;
\draw[line width=1pt,fill=black!20]
(1,0,1)--(1,0,0)--(1,1,0)--(1,1,1)--cycle
(1,1,1)--(1,1,0)--(1,2,0)--(1,2,1)--cycle
(1,2,1)--(1,2,0)--(1,3,0)--(1,3,1)--cycle
;
\draw[line width=1pt,fill=black!50]
(0,3,1)--(1,3,1)--(1,3,0)--(0,3,0)--cycle
;
\draw[line width=2pt](3,0,0)--(3,3,0)--(0,3,0)--(0,3,1)--(0,1,1)--(0,1,1)--(0,0,1)--(1,0,1)--(1,0,0)--(3,0,0)--cycle;
\end{scope}
\begin{scope}[xshift=5cm,yshift=0cm]
\clip(3,0,0)--(3,3,0)--(0,3,0)--(0,3,1)--(0,1,1)--(0,1,1)--(0,0,1)--(2,0,1)--(2,0,0)--(3,0,0)--cycle;
\draw[line width=1pt,fill=black!5]
(0,0,1)--(1,0,1)--(1,1,1)--(0,1,1)--cycle
(0,1,1)--(1,1,1)--(1,2,1)--(0,2,1)--cycle
(0,2,1)--(1,2,1)--(1,3,1)--(0,3,1)--cycle
(2,0,0)--(3,0,0)--(3,1,0)--(2,1,0)--cycle
(2,1,0)--(3,1,0)--(3,2,0)--(2,2,0)--cycle
(2,2,0)--(3,2,0)--(3,3,0)--(2,3,0)--cycle
;
\draw[line width=1pt,fill=yellow]
(1,1,1)--(2,1,1)--(2,2,1)--(1,2,1)--cycle
(1,2,1)--(2,2,1)--(2,3,1)--(1,3,1)--cycle
;
\draw[line width=1pt,fill=red]
(1,0,1)--(2,0,1)--(2,1,1)--(1,1,1)--cycle
;
\draw[line width=1pt,fill=black!20]
(2,0,1)--(2,0,0)--(2,1,0)--(2,1,1)--cycle
(2,1,1)--(2,1,0)--(2,2,0)--(2,2,1)--cycle
(2,2,1)--(2,2,0)--(2,3,0)--(2,3,1)--cycle
;
\draw[line width=1pt,fill=black!50]
(0,3,1)--(1,3,1)--(1,3,0)--(0,3,0)--cycle
(1,3,1)--(2,3,1)--(2,3,0)--(1,3,0)--cycle
;
\draw[line width=2pt](3,0,0)--(3,3,0)--(0,3,0)--(0,3,1)--(0,1,1)--(0,1,1)--(0,0,1)--(2,0,1)--(2,0,0)--(3,0,0)--cycle;
\end{scope}
\end{scope}
\end{scope}
\end{tikzpicture}
\end{align*}

Note that it is allowed to insert the remainder of $h$ in a different layer of $\pi$.
It is also allowed to cut and shift multiple times if needed.

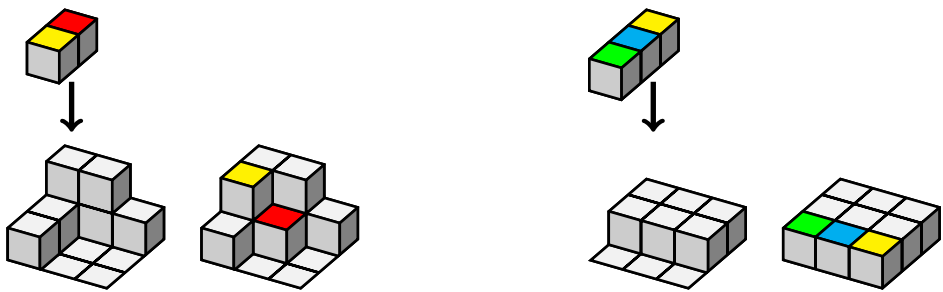
\begin{align*}
\begin{tikzpicture}[scale=.6,tdplot_main_coords]
\begin{scope}
\begin{scope}[xshift=0cm,yshift=-2cm]
\draw[line width=2pt,->](.5,.5,1.5)--(.5,.5,0);
\end{scope}
\begin{scope}[xshift=0cm,yshift=0cm]
\clip(2,0,0)--(2,1,0)--(0,1,0)--(0,1,1)--(0,0,1)--(2,0,1)--cycle;
\draw[line width=1pt,fill=red]
(0,0,1)--(1,0,1)--(1,1,1)--(0,1,1)--cycle
;
\draw[line width=1pt,fill=yellow]
(1,0,1)--(2,0,1)--(2,1,1)--(1,1,1)--cycle
;
\draw[line width=1pt,fill=black!20]
(2,0,1)--(2,0,0)--(2,1,0)--(2,1,1)--cycle
;
\draw[line width=1pt,fill=black!50]
(0,1,1)--(1,1,1)--(1,1,0)--(0,1,0)--cycle
(1,1,1)--(2,1,1)--(2,1,0)--(1,1,0)--cycle
;
\draw[line width=2pt](2,0,0)--(2,1,0)--(0,1,0)--(0,1,1)--(0,0,1)--(2,0,1)--cycle;
\end{scope}
\begin{scope}[yshift=-44mm]
\begin{scope}[xshift=0cm,yshift=0cm]
\clip(3,0,0)--(3,3,0)--(0,3,0)--(0,3,1)--(0,2,1)--(0,2,2)--(0,1,2)--(0,1,2)--(0,0,2)--(1,0,2)--(1,0,1)--(3,0,1)--cycle;
\draw[line width=1pt,fill=black!5]
(0,0,2)--(1,0,2)--(1,1,2)--(0,1,2)--cycle
(0,1,2)--(1,1,2)--(1,2,2)--(0,2,2)--cycle
(0,2,1)--(1,2,1)--(1,3,1)--(0,3,1)--cycle
(1,0,1)--(2,0,1)--(2,1,1)--(1,1,1)--cycle
(1,1,0)--(2,1,0)--(2,2,0)--(1,2,0)--cycle
(1,2,0)--(2,2,0)--(2,3,0)--(1,3,0)--cycle
(2,0,1)--(3,0,1)--(3,1,1)--(2,1,1)--cycle
(2,1,0)--(3,1,0)--(3,2,0)--(2,2,0)--cycle
(2,2,0)--(3,2,0)--(3,3,0)--(2,3,0)--cycle
;
\draw[line width=1pt,fill=black!20]
(1,0,2)--(1,0,1)--(1,1,1)--(1,1,2)--cycle
(1,1,2)--(1,1,1)--(1,2,1)--(1,2,2)--cycle
(1,1,1)--(1,1,0)--(1,2,0)--(1,2,1)--cycle
(1,2,1)--(1,2,0)--(1,3,0)--(1,3,1)--cycle
(3,0,1)--(3,0,0)--(3,1,0)--(3,1,1)--cycle
;
\draw[line width=1pt,fill=black!50]
(0,2,2)--(1,2,2)--(1,2,1)--(0,2,1)--cycle
(0,3,1)--(1,3,1)--(1,3,0)--(0,3,0)--cycle
(1,1,1)--(2,1,1)--(2,1,0)--(1,1,0)--cycle
(2,1,1)--(3,1,1)--(3,1,0)--(2,1,0)--cycle
;
\draw[line width=2pt](3,0,0)--(3,3,0)--(0,3,0)--(0,3,1)--(0,2,1)--(0,2,2)--(0,1,2)--(0,1,2)--(0,0,2)--(1,0,2)--(1,0,1)--(3,0,1)--cycle;
\end{scope}
\begin{scope}[xshift=5cm,yshift=0cm]
\clip(3,0,0)--(3,3,0)--(0,3,0)--(0,3,1)--(0,2,1)--(0,2,2)--(0,1,2)--(0,1,2)--(0,0,2)--(2,0,2)--(2,0,1)--(3,0,1)--cycle;
\draw[line width=1pt,fill=black!5]
(0,0,2)--(1,0,2)--(1,1,2)--(0,1,2)--cycle
(0,1,2)--(1,1,2)--(1,2,2)--(0,2,2)--cycle
(0,2,1)--(1,2,1)--(1,3,1)--(0,3,1)--cycle
(1,2,0)--(2,2,0)--(2,3,0)--(1,3,0)--cycle
(2,0,1)--(3,0,1)--(3,1,1)--(2,1,1)--cycle
(2,1,0)--(3,1,0)--(3,2,0)--(2,2,0)--cycle
(2,2,0)--(3,2,0)--(3,3,0)--(2,3,0)--cycle
;
\draw[line width=1pt,fill=yellow]
(1,0,2)--(2,0,2)--(2,1,2)--(1,1,2)--cycle
;
\draw[line width=1pt,fill=red]
(1,1,1)--(2,1,1)--(2,2,1)--(1,2,1)--cycle
;
\draw[line width=1pt,fill=black!20]
(1,1,2)--(1,1,1)--(1,2,1)--(1,2,2)--cycle
(1,2,1)--(1,2,0)--(1,3,0)--(1,3,1)--cycle
(2,0,2)--(2,0,1)--(2,1,1)--(2,1,2)--cycle
(2,1,1)--(2,1,0)--(2,2,0)--(2,2,1)--cycle
(3,0,1)--(3,0,0)--(3,1,0)--(3,1,1)--cycle
;
\draw[line width=1pt,fill=black!50]
(0,2,2)--(1,2,2)--(1,2,1)--(0,2,1)--cycle
(0,3,1)--(1,3,1)--(1,3,0)--(0,3,0)--cycle
(1,1,2)--(2,1,2)--(2,1,1)--(1,1,1)--cycle
(1,2,1)--(2,2,1)--(2,2,0)--(1,2,0)--cycle
(2,1,1)--(3,1,1)--(3,1,0)--(2,1,0)--cycle
;
\draw[line width=2pt](3,0,0)--(3,3,0)--(0,3,0)--(0,3,1)--(0,2,1)--(0,2,2)--(0,1,2)--(0,1,2)--(0,0,2)--(2,0,2)--(2,0,1)--(3,0,1)--cycle;
\end{scope}
\end{scope}
\end{scope}
\begin{scope}[xshift=15cm]
\begin{scope}[xshift=0cm,yshift=-2cm]
\draw[line width=2pt,->](.5,.5,1.5)--(.5,.5,0);
\end{scope}
\begin{scope}[xshift=0cm,yshift=0cm]
\clip(3,0,0)--(3,1,0)--(0,1,0)--(0,1,1)--(0,0,1)--(3,0,1)--cycle;
\draw[line width=1pt,fill=yellow]
(0,0,1)--(1,0,1)--(1,1,1)--(0,1,1)--cycle
;
\draw[line width=1pt,fill=cyan]
(1,0,1)--(2,0,1)--(2,1,1)--(1,1,1)--cycle
;
\draw[line width=1pt,fill=green]
(2,0,1)--(3,0,1)--(3,1,1)--(2,1,1)--cycle
;
\draw[line width=1pt,fill=black!20]
(3,0,1)--(3,0,0)--(3,1,0)--(3,1,1)--cycle
;
\draw[line width=1pt,fill=black!50]
(0,1,1)--(1,1,1)--(1,1,0)--(0,1,0)--cycle
(1,1,1)--(2,1,1)--(2,1,0)--(1,1,0)--cycle
(2,1,1)--(3,1,1)--(3,1,0)--(2,1,0)--cycle
;
\draw[line width=2pt](3,0,0)--(3,1,0)--(0,1,0)--(0,1,1)--(0,0,1)--(3,0,1)--cycle;
\end{scope}
\begin{scope}[yshift=-44mm]
\begin{scope}[xshift=0cm,yshift=0cm]
\clip(3,0,0)--(3,3,0)--(0,3,0)--(0,3,1)--(0,1,1)--(0,1,1)--(0,0,1)--(2,0,1)--(2,0,0)--(3,0,0)--cycle;
\draw[line width=1pt,fill=black!5]
(0,0,1)--(1,0,1)--(1,1,1)--(0,1,1)--cycle
(0,1,1)--(1,1,1)--(1,2,1)--(0,2,1)--cycle
(0,2,1)--(1,2,1)--(1,3,1)--(0,3,1)--cycle
(2,0,0)--(3,0,0)--(3,1,0)--(2,1,0)--cycle
(2,1,0)--(3,1,0)--(3,2,0)--(2,2,0)--cycle
(2,2,0)--(3,2,0)--(3,3,0)--(2,3,0)--cycle
(1,1,1)--(2,1,1)--(2,2,1)--(1,2,1)--cycle
(1,2,1)--(2,2,1)--(2,3,1)--(1,3,1)--cycle
(1,0,1)--(2,0,1)--(2,1,1)--(1,1,1)--cycle
;
\draw[line width=1pt,fill=black!20]
(2,0,1)--(2,0,0)--(2,1,0)--(2,1,1)--cycle
(2,1,1)--(2,1,0)--(2,2,0)--(2,2,1)--cycle
(2,2,1)--(2,2,0)--(2,3,0)--(2,3,1)--cycle
;
\draw[line width=1pt,fill=black!50]
(0,3,1)--(1,3,1)--(1,3,0)--(0,3,0)--cycle
(1,3,1)--(2,3,1)--(2,3,0)--(1,3,0)--cycle
;
\draw[line width=2pt](3,0,0)--(3,3,0)--(0,3,0)--(0,3,1)--(0,1,1)--(0,1,1)--(0,0,1)--(2,0,1)--(2,0,0)--(3,0,0)--cycle;
\end{scope}
\begin{scope}[xshift=5cm,yshift=0cm]
\clip(3,0,0)--(3,3,0)--(0,3,0)--(0,3,1)--(0,1,1)--(0,1,1)--(0,0,1)--(3,0,1)--cycle;
\draw[line width=1pt,fill=black!5]
(0,0,1)--(1,0,1)--(1,1,1)--(0,1,1)--cycle
(0,1,1)--(1,1,1)--(1,2,1)--(0,2,1)--cycle
(0,2,1)--(1,2,1)--(1,3,1)--(0,3,1)--cycle
(1,0,1)--(2,0,1)--(2,1,1)--(1,1,1)--cycle
(1,1,1)--(2,1,1)--(2,2,1)--(1,2,1)--cycle
(1,2,1)--(2,2,1)--(2,3,1)--(1,3,1)--cycle
;
\draw[line width=1pt,fill=green]
(2,0,1)--(3,0,1)--(3,1,1)--(2,1,1)--cycle
;
\draw[line width=1pt,fill=cyan]
(2,1,1)--(3,1,1)--(3,2,1)--(2,2,1)--cycle
;
\draw[line width=1pt,fill=yellow]
(2,2,1)--(3,2,1)--(3,3,1)--(2,3,1)--cycle
;
\draw[line width=1pt,fill=black!20]
(3,0,1)--(3,0,0)--(3,1,0)--(3,1,1)--cycle
(3,1,1)--(3,1,0)--(3,2,0)--(3,2,1)--cycle
(3,2,1)--(3,2,0)--(3,3,0)--(3,3,1)--cycle
;
\draw[line width=1pt,fill=black!50]
(0,3,1)--(1,3,1)--(1,3,0)--(0,3,0)--cycle
(1,3,1)--(2,3,1)--(2,3,0)--(1,3,0)--cycle
(2,3,1)--(3,3,1)--(3,3,0)--(2,3,0)--cycle
;
\draw[line width=2pt](3,0,0)--(3,3,0)--(0,3,0)--(0,3,1)--(0,1,1)--(0,1,1)--(0,0,1)--(3,0,1)--cycle;
\end{scope}
\end{scope}
\end{scope}
\end{tikzpicture}
\end{align*}

If this (iterated) process yields a reverse plane partition in the end, it is denoted by $h*\pi$ and we say $h$ inserts into $\pi$.
On the other hand it might also happen that the procedure fails, as in the following examples.
In this case we say $h$ does not insert into $\pi$.
Note in particular the third example, in which insertion fails since we demand that each part of $h$ remains intact and cannot be cut in two.

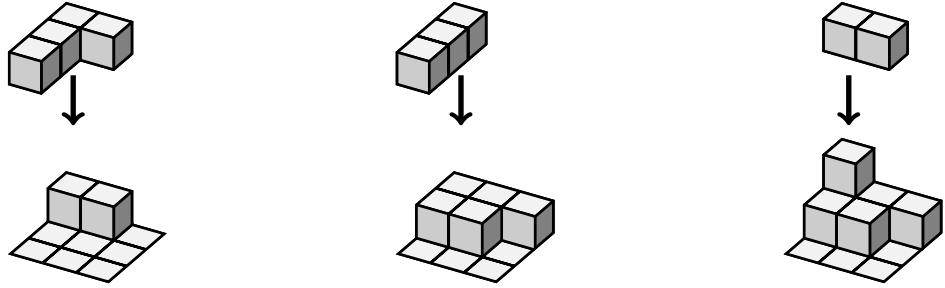
\begin{align*}
\begin{tikzpicture}[scale=.6,tdplot_main_coords]
\begin{scope}[xshift=0cm]
\begin{scope}[xshift=0cm,yshift=-2cm]
\draw[line width=2pt,->](.5,.5,1.5)--(.5,.5,0);
\end{scope}
\begin{scope}[xshift=0cm,yshift=0cm]
\clip(3,0,0)--(3,1,0)--(1,1,0)--(1,2,0)--(0,2,0)--(0,2,1)--(0,1,1)--(0,1,1)--(0,0,1)--(3,0,1)--cycle;
\draw[line width=1pt,fill=black!5]
(0,0,1)--(1,0,1)--(1,1,1)--(0,1,1)--cycle
(0,1,1)--(1,1,1)--(1,2,1)--(0,2,1)--cycle
(1,0,1)--(2,0,1)--(2,1,1)--(1,1,1)--cycle
(2,0,1)--(3,0,1)--(3,1,1)--(2,1,1)--cycle
;
\draw[line width=1pt,fill=black!20]
(1,1,1)--(1,1,0)--(1,2,0)--(1,2,1)--cycle
(3,0,1)--(3,0,0)--(3,1,0)--(3,1,1)--cycle
;
\draw[line width=1pt,fill=black!50]
(0,2,1)--(1,2,1)--(1,2,0)--(0,2,0)--cycle
(1,1,1)--(2,1,1)--(2,1,0)--(1,1,0)--cycle
(2,1,1)--(3,1,1)--(3,1,0)--(2,1,0)--cycle
;
\draw[line width=2pt](3,0,0)--(3,1,0)--(1,1,0)--(1,2,0)--(0,2,0)--(0,2,1)--(0,1,1)--(0,1,1)--(0,0,1)--(3,0,1)--cycle;
\end{scope}
\begin{scope}[xshift=0cm,yshift=-44mm]
\clip(3,0,0)--(3,3,0)--(0,3,0)--(0,3,0)--(0,2,0)--(0,2,1)--(0,1,1)--(0,1,1)--(0,0,1)--(1,0,1)--(1,0,0)--(3,0,0)--cycle;
\draw[line width=1pt,fill=black!5]
(0,0,1)--(1,0,1)--(1,1,1)--(0,1,1)--cycle
(0,1,1)--(1,1,1)--(1,2,1)--(0,2,1)--cycle
(0,2,0)--(1,2,0)--(1,3,0)--(0,3,0)--cycle
(1,0,0)--(2,0,0)--(2,1,0)--(1,1,0)--cycle
(1,1,0)--(2,1,0)--(2,2,0)--(1,2,0)--cycle
(1,2,0)--(2,2,0)--(2,3,0)--(1,3,0)--cycle
(2,0,0)--(3,0,0)--(3,1,0)--(2,1,0)--cycle
(2,1,0)--(3,1,0)--(3,2,0)--(2,2,0)--cycle
(2,2,0)--(3,2,0)--(3,3,0)--(2,3,0)--cycle
;
\draw[line width=1pt,fill=black!20]
(1,0,1)--(1,0,0)--(1,1,0)--(1,1,1)--cycle
(1,1,1)--(1,1,0)--(1,2,0)--(1,2,1)--cycle
;
\draw[line width=1pt,fill=black!50]
(0,2,1)--(1,2,1)--(1,2,0)--(0,2,0)--cycle
;
\draw[line width=2pt](3,0,0)--(3,3,0)--(0,3,0)--(0,3,0)--(0,2,0)--(0,2,1)--(0,1,1)--(0,1,1)--(0,0,1)--(1,0,1)--(1,0,0)--(3,0,0)--cycle;
\end{scope}
\end{scope}
\begin{scope}[xshift=10cm]
\begin{scope}[xshift=0cm,yshift=-2cm]
\draw[line width=2pt,->](.5,.5,1.5)--(.5,.5,0);
\end{scope}
\begin{scope}[xshift=0cm,yshift=0cm]
\clip(3,0,0)--(3,1,0)--(0,1,0)--(0,1,1)--(0,0,1)--(3,0,1)--cycle;
\draw[line width=1pt,fill=black!5]
(0,0,1)--(1,0,1)--(1,1,1)--(0,1,1)--cycle
(1,0,1)--(2,0,1)--(2,1,1)--(1,1,1)--cycle
(2,0,1)--(3,0,1)--(3,1,1)--(2,1,1)--cycle
;
\draw[line width=1pt,fill=black!20]
(3,0,1)--(3,0,0)--(3,1,0)--(3,1,1)--cycle
;
\draw[line width=1pt,fill=black!50]
(0,1,1)--(1,1,1)--(1,1,0)--(0,1,0)--cycle
(1,1,1)--(2,1,1)--(2,1,0)--(1,1,0)--cycle
(2,1,1)--(3,1,1)--(3,1,0)--(2,1,0)--cycle
;
\draw[line width=2pt](3,0,0)--(3,1,0)--(0,1,0)--(0,1,1)--(0,0,1)--(3,0,1)--cycle;
\end{scope}
\begin{scope}[xshift=0cm,yshift=-44mm]
\clip(3,0,0)--(3,3,0)--(0,3,0)--(0,3,1)--(0,1,1)--(0,1,1)--(0,0,1)--(2,0,1)--(2,0,0)--(3,0,0)--cycle;
\draw[line width=1pt,fill=black!5]
(0,0,1)--(1,0,1)--(1,1,1)--(0,1,1)--cycle
(0,1,1)--(1,1,1)--(1,2,1)--(0,2,1)--cycle
(0,2,1)--(1,2,1)--(1,3,1)--(0,3,1)--cycle
(1,0,1)--(2,0,1)--(2,1,1)--(1,1,1)--cycle
(1,1,1)--(2,1,1)--(2,2,1)--(1,2,1)--cycle
(1,2,0)--(2,2,0)--(2,3,0)--(1,3,0)--cycle
(2,0,0)--(3,0,0)--(3,1,0)--(2,1,0)--cycle
(2,1,0)--(3,1,0)--(3,2,0)--(2,2,0)--cycle
(2,2,0)--(3,2,0)--(3,3,0)--(2,3,0)--cycle
;
\draw[line width=1pt,fill=black!20]
(1,2,1)--(1,2,0)--(1,3,0)--(1,3,1)--cycle
(2,0,1)--(2,0,0)--(2,1,0)--(2,1,1)--cycle
(2,1,1)--(2,1,0)--(2,2,0)--(2,2,1)--cycle
;
\draw[line width=1pt,fill=black!50]
(0,3,1)--(1,3,1)--(1,3,0)--(0,3,0)--cycle
(1,2,1)--(2,2,1)--(2,2,0)--(1,2,0)--cycle
;
\draw[line width=2pt](3,0,0)--(3,3,0)--(0,3,0)--(0,3,1)--(0,1,1)--(0,1,1)--(0,0,1)--(2,0,1)--(2,0,0)--(3,0,0)--cycle;
\end{scope}
\end{scope}
\begin{scope}[xshift=20cm]
\begin{scope}[xshift=0cm,yshift=-2cm]
\draw[line width=2pt,->](.5,.5,1.5)--(.5,.5,0);
\end{scope}
\begin{scope}[xshift=0cm,yshift=0cm]
\clip(1,0,0)--(1,2,0)--(0,2,0)--(0,2,1)--(0,1,1)--(0,1,1)--(0,0,1)--(1,0,1)--cycle;
\draw[line width=1pt,fill=black!5]
(0,0,1)--(1,0,1)--(1,1,1)--(0,1,1)--cycle
(0,1,1)--(1,1,1)--(1,2,1)--(0,2,1)--cycle
;
\draw[line width=1pt,fill=black!20]
(1,0,1)--(1,0,0)--(1,1,0)--(1,1,1)--cycle
(1,1,1)--(1,1,0)--(1,2,0)--(1,2,1)--cycle
;
\draw[line width=1pt,fill=black!50]
(0,2,1)--(1,2,1)--(1,2,0)--(0,2,0)--cycle
;
\draw[line width=2pt](1,0,0)--(1,2,0)--(0,2,0)--(0,2,1)--(0,1,1)--(0,1,1)--(0,0,1)--(1,0,1)--cycle;
\end{scope}
\begin{scope}[xshift=0cm,yshift=-44mm]
\clip(3,0,0)--(3,3,0)--(0,3,0)--(0,3,1)--(0,1,1)--(0,1,2)--(0,0,2)--(1,0,2)--(1,0,1)--(2,0,1)--(2,0,0)--(3,0,0)--cycle;
\draw[line width=1pt,fill=black!5]
(0,0,2)--(1,0,2)--(1,1,2)--(0,1,2)--cycle
(0,1,1)--(1,1,1)--(1,2,1)--(0,2,1)--cycle
(0,2,1)--(1,2,1)--(1,3,1)--(0,3,1)--cycle
(1,0,1)--(2,0,1)--(2,1,1)--(1,1,1)--cycle
(1,1,1)--(2,1,1)--(2,2,1)--(1,2,1)--cycle
(1,2,0)--(2,2,0)--(2,3,0)--(1,3,0)--cycle
(2,0,0)--(3,0,0)--(3,1,0)--(2,1,0)--cycle
(2,1,0)--(3,1,0)--(3,2,0)--(2,2,0)--cycle
(2,2,0)--(3,2,0)--(3,3,0)--(2,3,0)--cycle
;
\draw[line width=1pt,fill=black!20]
(1,0,2)--(1,0,1)--(1,1,1)--(1,1,2)--cycle
(1,2,1)--(1,2,0)--(1,3,0)--(1,3,1)--cycle
(2,0,1)--(2,0,0)--(2,1,0)--(2,1,1)--cycle
(2,1,1)--(2,1,0)--(2,2,0)--(2,2,1)--cycle
;
\draw[line width=1pt,fill=black!50]
(0,1,2)--(1,1,2)--(1,1,1)--(0,1,1)--cycle
(0,3,1)--(1,3,1)--(1,3,0)--(0,3,0)--cycle
(1,2,1)--(2,2,1)--(2,2,0)--(1,2,0)--cycle
;
\draw[line width=2pt](3,0,0)--(3,3,0)--(0,3,0)--(0,3,1)--(0,1,1)--(0,1,2)--(0,0,2)--(1,0,2)--(1,0,1)--(2,0,1)--(2,0,0)--(3,0,0)--cycle;
\end{scope}
\end{scope}
\end{tikzpicture}
\end{align*}

The main results of this extended abstract can be phrased as follows.
First of all every reverse plane partition can be built as described above using only bricks of rim-hook shape as building blocks.
Secondly, given a multi-set of bricks there is always a way to sort them (lexicographically) such that they can be successively inserted into the zero reverse plane partition.
Thirdly, each reverse plane partition can be built in a unique way such that all building blocks are inserted in lexicographic order.
Altogether we obtain a bijective correspondence between reverse plane partitions of a given shape and the multi-sets of rim-hooks of the same shape.

\section{Inserting rim-hooks}\label{Section:insert}

In this section we first fix the needed notation concerning partitions and reverse plane partitions.
Afterwards a precise formulation of the insertion algorithm described in Section~\ref{Section:blocks} is given.
Furthermore, we state the main results of this abstract in Theorems~\ref{Theorem:lexinsert} and~\ref{Theorem:lexfact}.

\smallskip
Let $\N=\{0,1,2,\dots\}$ denote the non-negative integers.
Moreover for $n\in\N$ set $[n]=\{1,\dots,n\}$.
A pair $u\in\Z^2$ is called \emph{cell}.
For each cell $u$ denote the \emph{North}, \emph{East}, \emph{South} and \emph{West neighbours} of $u$ by
\begin{align*}
\n u=(u_1-1,u_2),\qquad
\e u=(u_1,u_2+1),\qquad
\s u=(u_1+1,u_2),\qquad
\w u=(u_1,u_2-1).
\end{align*}
A \emph{partition} $\lambda$ is a weakly decreasing sequence $\lambda_1\geq\lambda_2\geq\dots\geq\lambda_r>0$ of positive integers.
The elements $\lambda_i$ are called \emph{parts} of the partition.
The number of positive parts is called the \emph{length} of the partition and is denoted by $\ell(\lambda)$.
We identify each partition with a set of cells $\lambda=\{(i,j):i\in[\ell(\lambda)],j\in[\lambda_{i}]\}$ called the \emph{Young diagram} of $\lambda$.

A \emph{reverse plane partition} of shape $\lambda$ is a map $\pi:\lambda\to\N$ such that $\pi(u)\leq\pi(\e u)$ and $\pi(u)\leq\pi(\s u)$ for all $u\in\lambda$.
By convention $\pi(u)=0$ if $u_1\leq0$ or $u_2\leq0$ and $\pi(u)=\infty$ if $u_1,u_2\geq1$ but $u\notin\lambda$.
Example \eqref{eq:0123-122-1} shows the partition $\lambda=(4,3,1)$, a reverse plane partition $\pi$ of shape $\lambda$ and the representation of $\pi$ as an arrangement of stacks of cubes.

The \emph{conjugate} of a partition $\lambda$ is the partition $\lambda'=\{(j,i):(i,j)\in\lambda\}$.
The \emph{hook} $H(u)$ of a cell $u\in\lambda$ consists of the cell $u$ itself and those cells $v\in\lambda$ that lie directly East of $u$ or directly South of $u$.
The \emph{hook-length} $h(u)=\lambda_i+\lambda_j'-i-j+1$ denotes the cardinality of $H(u)$.
For example, in the partition $\lambda=(4,3,1)$ above we have $H(1,2)=\{(1,2),(1,3),(1,4),(2,2)\}$ and thus $h(1,2)=4$.
%

The \emph{content} of a cell $u$ is defined as
\begin{align*}
c(u)=u_2-u_1.
\end{align*}
Equip $\lambda$ with two total orders: the \emph{reverse lexicographic order} $(\lambda,\leq)$ and the \emph{content order} $(\lambda,\trianglelefteq)$.
Let $u,v\in\lambda$ then $u\leq v$ if and only if $u_2>v_2$ or $u_2=v_2$ and $u_1\geq v_1$.
Moreover $u\trianglelefteq v$ if and only if $c(u)>c(v)$ or $c(u)=c(v)$ and $u_1\geq v_1$.
In the example below the reverse lexicographic order (left) and content order (right) are illustrated for our running example $\lambda=(4,3,1)$.

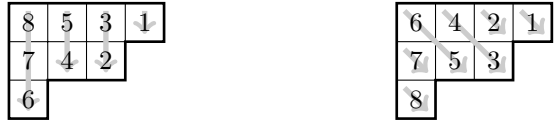
\begin{align*}
\begin{tikzpicture}[scale=.6]
\begin{scope}[xshift=0cm]
\draw
(0,1)--(1,1)
(0,2)--(3,2)
(1,1)--(1,3)
(2,1)--(2,3)
(3,2)--(3,3)
;
\draw[black!20,line width=2pt,->](.5,2.8)--(.5,.2);
\draw[black!20,line width=2pt,->](1.5,2.8)--(1.5,1.2);
\draw[black!20,line width=2pt,->](2.5,2.8)--(2.5,1.2);
\draw[black!20,line width=2pt,->](3.5,2.8)--(3.5,2.2)
;
\draw[line width=1pt](0,0)--(1,0)--(1,1)--(3,1)--(3,2)--(4,2)--(4,3)--(0,3)--cycle;
\draw[xshift=5mm,yshift=5mm]
(0,2)node{{$8$}}
(1,2)node{{$5$}}
(2,2)node{{$3$}}
(3,2)node{{$1$}}
(0,1)node{{$7$}}
(1,1)node{{$4$}}
(2,1)node{{$2$}}
(0,0)node{{$6$}};
\end{scope}
\begin{scope}[xshift=10cm]
\draw
(0,1)--(1,1)
(0,2)--(3,2)
(1,1)--(1,3)
(2,1)--(2,3)
(3,2)--(3,3)
;
\draw[black!20,line width=2pt,->](.2,.8)--(.8,.2);
\draw[black!20,line width=2pt,->](.2,1.8)--(.8,1.2);
\draw[black!20,line width=2pt,->](.2,2.8)--(1.8,1.2);
\draw[black!20,line width=2pt,->](1.2,2.8)--(2.8,1.2);
\draw[black!20,line width=2pt,->](2.2,2.8)--(2.8,2.2);
\draw[black!20,line width=2pt,->](3.2,2.8)--(3.8,2.2)
;
\draw[line width=1pt](0,0)--(1,0)--(1,1)--(3,1)--(3,2)--(4,2)--(4,3)--(0,3)--cycle;
\draw[xshift=5mm,yshift=5mm]
(0,2)node{{$6$}}
(1,2)node{{$4$}}
(2,2)node{{$2$}}
(3,2)node{{$1$}}
(0,1)node{{$7$}}
(1,1)node{{$5$}}
(2,1)node{{$3$}}
(0,0)node{{$8$}};
\end{scope}
\end{tikzpicture}
\end{align*}

A \emph{North-East-path} in $\lambda$ is a sequence $P=(u_0,u_1,\dots,u_{s})$ of cells $u_k\in\lambda$ such that $u_{k}\in\{\n u_{k-1},\e u_{k-1}\}$ for all $k\in[s]$.
We call $\ell(P)=s$ the \emph{length}, $\alpha(P)=u_0$ the \emph{head} and $\omega(P)=u_{s}$ the \emph{tail} of the path $P$.
Sometimes it is more convenient to consider South-West-paths instead.
A \emph{South-West-path} in $\lambda$ is a sequence $Q=(v_0,v_1,\dots,v_{s})$ of cells $v_k\in\lambda$ such that $v_k\in\{\s v_{k-1},\w v_{k-1}\}$ for all $k\in[s]$.
Denote by $P'=(u_{s},\dots,u_1,u_0)$ the \emph{reverse path} of $P$.
Clearly the reverse of a South-West-path is a North-East-path and vice versa.
Given a South-West-path $Q$ set $\ell(Q)=s$, $\alpha(Q)=v_s$ and $\omega(Q)=v_0$ such that all notions are independent of the fact whether $Q$ is regarded as a North-East-path or as a South-West-path.
That is, $\alpha(P)=\alpha(P')$, $\omega(P)=\omega(P')$ and so forth.

%

A \emph{rim-hook} of $\lambda$ is a North-East-path $h$ in $\lambda$ such that $\s\alpha(h)\notin\lambda$, $\e\omega(h)\notin\lambda$ and $\e\s u\notin\lambda$ for all $u\in h$.
For each cell $(i,j)\in\lambda$ there is a rim-hook $h$ with $\alpha(h)=(\lambda_j',j)$ and $\omega(h)=(i,\lambda_i)$.
This correspondence is a bijection between $\lambda$ and the set of rim-hooks of $\lambda$.
We use this correspondence to transfer the reverse lexicographic order on cells to rim-hooks.
Thereby for given rim-hooks $f$ and $h$ we have $f\leq h$ if and only if either $c(\alpha(f))>c(\alpha(h))$ or $c(\alpha(f))=c(\alpha(h))$ and $c(\omega(f))\leq c(\omega(h))$.
The example \eqref{eq:rimhooks} shows all rim-hooks of the partition $\lambda=(4,3,1)$ in reverse lexicographic order starting with the minimum.

A cell $u\in\lambda$ is called \emph{outer corner} if $\e u,\s u\notin\lambda$ and \emph{inner corner} if $\e u,\s u\in\lambda$ but $\e\s u\notin\lambda$.
Let $i_1,\dots,i_r$ be the contents of the inner corners of $\lambda$ and $o_1,\dots,o_{r+1}$ be the contents of the outer corners of $\lambda$ sorted such that
\begin{align*}
o_1<i_1<o_2<\dots<o_r<i_r<o_{r+1}.
\end{align*}
Divide $\lambda$ into four regions.
\begin{align*}
\begin{split}
\I&=\{u\in\lambda:c(u)=i_k\text{ for some }k\in[r]\}\\
\O&=\{u\in\lambda:c(u)=o_k\text{ for some }k\in[r+1]\}\\
\A&=\{u\in\lambda:c(u)<o_1\text{ or }i_k<c(u)<o_{k+1}\text{ for some }k\in[r]\}\\
\B&=\{u\in\lambda:o_k<c(u)<i_k\text{ for some }k\in[r]\text{ or }o_{r+1}<c(u)\}
\end{split}
\end{align*}
The example below shows the region $\I$ in yellow and the region $\O$ in orange.
The contents of the inner corners are $-2,2,4$ the contents of the outer corners are $-5,1,3,7$.

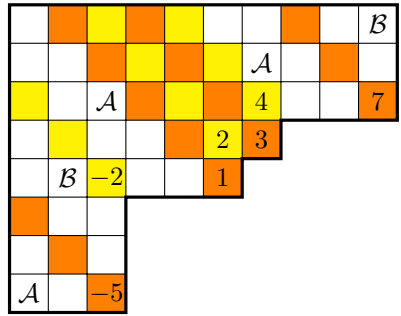
\begin{align*}
\begin{tikzpicture}[scale=.6]
\fill[orange]
(0,2)--(1,2)--(1,1)--(2,1)--(2,0)--(3,0)--(3,1)--(2,1)--(2,2)--(1,2)--(1,3)--(0,3)--cycle
(1,7)--(2,7)--(2,6)--(3,6)--(3,5)--(4,5)--(4,4)--(5,4)--(5,3)--(6,3)--(6,4)--(7,4)--(7,5)--(6,5)--(6,6)--(5,6)--(5,7)--(4,7)--(4,8)--(1,8)--cycle
(7,7)--(8,7)--(8,6)--(9,6)--(9,5)--(10,5)--(10,6)--(9,6)--(9,7)--(8,7)--(8,8)--(7,8)--cycle
;
\fill[yellow]
(0,5)--(1,5)--(1,4)--(2,4)--(2,3)--(3,3)--(3,4)--(2,4)--(2,5)--(1,5)--(1,6)--(0,6)--cycle
(2,7)--(3,7)--(3,6)--(4,6)--(4,5)--(5,5)--(5,4)--(6,4)--(6,5)--(5,5)--(5,6)--(4,6)--(4,7)--(3,7)--(3,8)--(2,8)--cycle
(4,7)--(5,7)--(5,6)--(6,6)--(6,5)--(7,5)--(7,6)--(6,6)--(6,7)--(5,7)--(5,8)--(4,8)--cycle
;
\draw
(0,1)--(3,1)
(0,2)--(3,2)
(0,3)--(3,3)
(0,4)--(6,4)
(0,5)--(7,5)
(0,6)--(10,6)
(0,7)--(10,7)
(1,8)--(1,0)
(2,8)--(2,0)
(3,8)--(3,3)
(4,8)--(4,3)
(5,8)--(5,3)
(6,8)--(6,4)
(7,8)--(7,5)
(8,8)--(8,5)
(9,8)--(9,5)
;
\draw[very thick]
(0,0)--(3,0)--(3,3)--(6,3)--(6,4)--(7,4)--(7,5)--(10,5)--(10,8)--(0,8)--cycle
;
\draw[xshift=5mm,yshift=5mm]
(2,0)node{{$-5$}}
(2,3)node{{$-2$}}
(5,3)node{{$1$}}
(5,4)node{{$2$}}
(6,4)node{{$3$}}
(6,5)node{{$4$}}
(9,5)node{{$7$}}
(0,0)node{{$\A$}}
(2,5)node{{$\A$}}
(6,6)node{{$\A$}}
(1,3)node{{$\B$}}
(9,7)node{{$\B$}}
;
\end{tikzpicture}
\end{align*}

The motivation for these definitions is as follows.
Note that $\alpha(h)\in\A\cup\O$ and $\omega(h)\in\O\cup\B$ for any rim-hook of $\lambda$.
More precisely if $h$ is a rim-hook of $\lambda$ and $u\in h$ with $u\in\I\cup\A$ then $\e u\in h$.
If $u\in h$ with $u\in\B\cup\I$ then also $\s u\in h$.
We consider paths that are similar to rim-hooks in the sense that they fulfil similar properties.

\smallskip
Next let $\pi$ be a reverse plane partition of shape $\lambda$ and $P$ a path in $\lambda$.
Define a map $\pi\pm P:\lambda\to\Z$ via
\begin{align*}
(\pi\pm P)(u)=
\begin{cases}
\pi(u)\pm1&\quad\text{if }u\in P,\\
\pi(u)&\quad\text{otherwise.}\\
\end{cases}
\end{align*}
By definition $P$ is \emph{$\pi$-compatible} if the following two conditions are satisfied:
\begin{align}\label{eq:shh}
\text{If }u\in P\text{ and }u\in\I\cup\A&\text{ then }\e u\in P\text{ and }\pi(\e u)=\pi(u).\\ \label{eq:shp}
\text{If }u,\n u\in P&\text{ then }\pi(\n u)=\pi(u).
\end{align}
A rim-hook $h$ \emph{inserts} into $\pi$ if there exists a $\pi$-compatible path $P$ in $\lambda$ such that $\omega(P)=\omega(h)$ and $\ell(P)=\ell(h)$ and $\pi+P$ is again a reverse plane partition.
It is not difficult to show that this path $P$ is unique if it exists at all.
Thus if $h$ inserts into $\pi$ we may define $h*\pi=\pi+P$.

Alternatively construct a South-West-path $P(\pi,h)$ with $\omega(P(\pi,h))=\omega(h)$ as follows.
While $\ell(P(\pi,h))<\ell(h)$, if $u$ is the current cell then move to
\begin{align}\label{eq:P}
\text{the cell}\quad
\begin{cases}
\s u&\quad\text{if }u\in\B\cup\I\text{ and }\pi(\s u)=\pi(u),\\
\w u&\quad\text{otherwise.}
\end{cases}
\end{align}
Then $P(\pi,h)$ satisfies $\omega(P(\pi,h))=\omega(h)$, $\ell(P(\pi,h))=\ell(h)$ and \eqref{eq:shp} by definition.
One can show that $h$ inserts into $\pi$ if and only if $P(\pi,h)$ satisfies \eqref{eq:shh} and $\pi+P(\pi,h)$ is a reverse plane partition.

These definitions should be compared with the examples of the previous section.

\smallskip
We now formulate two theorems essentially stating that our insertion algorithm yields a bijection.
The proofs, which are omitted in this extended abstract, are slightly more involved than the proof of Hillman and Grassl.
Among other arguments they rely on a precise formulation of the familiar fact that certain paths in $\lambda$ cannot cross.
The first theorem gives a sufficient criterion for when a rim-hook inserts into a reverse plane partition.

\begin{thm}\label{Theorem:lexinsert} Let $\lambda$ be a partition and $h_1,h_2,\dots,h_s$ a sequence of rim-hooks of $\lambda$ such that $h_i\leq h_{i+1}$ for all $i\in[s-1]$.
Then $h_i$ inserts into $h_{i+1}*\dots*h_s*0$ for all $i\in[s]$, where $0$ denotes the zero reverse plane partition.
\end{thm}

Let $h_1,h_2,\dots,h_s$ be rim-hooks such that $h_i\leq h_{i+1}$ for all $i\in[s-1]$, and set $\pi=h_1*h_2*\dots*h_s*0$.
The sequence $h_1,h_2,\dots,h_s$ is called \emph{lexicographic factorisation} of $\pi$.
By Theorem~\ref{Theorem:lexinsert} there is a well-defined map $\Psi$ assigning to each multi-set of rim-hooks of $\lambda$ a reverse plane partition of shape $\lambda$ by inserting the rim-hooks into the zero reverse plane partition in lexicographic order.
The next result states that this map is a bijection, that is, each reverse plane partition has a unique lexicographic factorisation.

\begin{thm}\label{Theorem:lexfact} Let $\lambda$ be a partition and $\pi$ a reverse plane partition of shape $\lambda$.
Then there exists a unique sequence $h_1,h_2,\dots,h_s$ of rim-hooks of $\lambda$ such that $h_i\leq h_{i+1}$ for all $i\in[s-1]$ 
and
\begin{align*}
h_1*h_2*\dots*h_s*0=\pi.
\end{align*}
\end{thm}

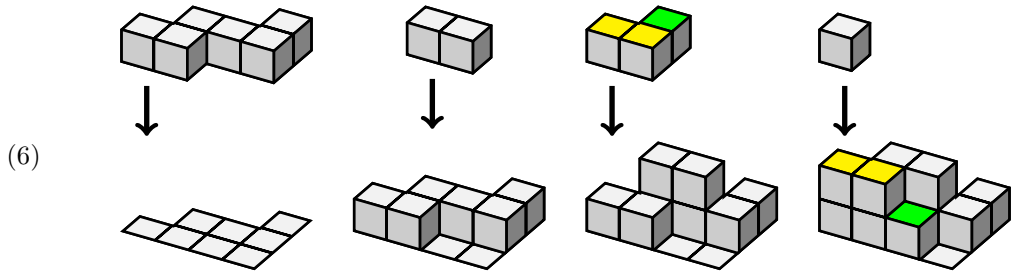
\begin{align}
\label{eq:0123-122-1(lex)}
\raisebox{-4em}{
\begin{tikzpicture}[scale=.6,tdplot_main_coords]
\begin{scope}[xshift=0cm,yshift=-2cm]
\draw[line width=2pt,->](2.5,.5,1.5)--(2.5,.5,0);
\begin{scope}[xshift=6cm,yshift=0cm]
\draw[line width=2pt,->](1.5,1.5,1.5)--(1.5,1.5,0);
\end{scope}
\begin{scope}[xshift=12cm,yshift=0cm]
\draw[line width=2pt,->](2.5,.5,1.5)--(2.5,.5,0);
\end{scope}
\begin{scope}[xshift=18cm,yshift=0cm]
\draw[line width=2pt,->](2.5,.5,1.5)--(2.5,.5,0);
\end{scope}
\end{scope}
\begin{scope}[xshift=0cm,yshift=0cm]
\clip(3,0,0)--(3,1,0)--(3,2,0)--(2,2,0)--(2,3,0)--(2,4,0)--(1,4,0)--(0,4,0)--(0,4,1)--(0,3,1)--(1,3,1)--(1,2,1)--(1,1,1)--(2,1,1)--(2,0,1)--(3,0,1)--(3,0,0)--cycle;
\draw[line width=1pt,fill=black!5]
(2,1,1)--(3,1,1)--(3,2,1)--(2,2,1)--cycle
(2,0,1)--(3,0,1)--(3,1,1)--(2,1,1)--cycle
(1,3,1)--(2,3,1)--(2,4,1)--(1,4,1)--cycle
(1,2,1)--(2,2,1)--(2,3,1)--(1,3,1)--cycle
(1,1,1)--(2,1,1)--(2,2,1)--(1,2,1)--cycle
(0,3,1)--(1,3,1)--(1,4,1)--(0,4,1)--cycle
;
\draw[line width=1pt,fill=black!20]
(3,1,0)--(3,2,0)--(3,2,1)--(3,1,1)--cycle
(3,0,0)--(3,1,0)--(3,1,1)--(3,0,1)--cycle
(2,3,0)--(2,4,0)--(2,4,1)--(2,3,1)--cycle
(2,2,0)--(2,3,0)--(2,3,1)--(2,2,1)--cycle
;
\draw[line width=1pt,fill=black!50]
(2,2,0)--(2,2,1)--(3,2,1)--(3,2,0)--cycle
(1,4,0)--(1,4,1)--(2,4,1)--(2,4,0)--cycle
(0,4,0)--(0,4,1)--(1,4,1)--(1,4,0)--cycle
;
\draw[line width=2pt](3,0,0)--(3,1,0)--(3,2,0)--(2,2,0)--(2,3,0)--(2,4,0)--(1,4,0)--(0,4,0)--(0,4,1)--(0,3,1)--(1,3,1)--(1,2,1)--(1,1,1)--(2,1,1)--(2,0,1)--(3,0,1)--(3,0,0)--cycle;
\end{scope}
\begin{scope}[xshift=6cm,yshift=0cm]
\clip(2,1,0)--(2,2,0)--(2,3,0)--(1,3,0)--(1,3,1)--(1,2,1)--(1,1,1)--(2,1,1)--cycle;
\draw[line width=1pt,fill=black!5]
(1,2,1)--(2,2,1)--(2,3,1)--(1,3,1)--cycle
(1,1,1)--(2,1,1)--(2,2,1)--(1,2,1)--cycle
;
\draw[line width=1pt,fill=black!20]
(2,2,0)--(2,3,0)--(2,3,1)--(2,2,1)--cycle
(2,1,0)--(2,2,0)--(2,2,1)--(2,1,1)--cycle
;
\draw[line width=1pt,fill=black!50]
(1,3,0)--(1,3,1)--(2,3,1)--(2,3,0)--cycle
;
\draw[line width=2pt](2,1,0)--(2,2,0)--(2,3,0)--(1,3,0)--(1,3,1)--(1,2,1)--(1,1,1)--(2,1,1)--cycle;
\end{scope}
\begin{scope}[xshift=12cm,yshift=0cm]
\clip(3,0,0)--(3,1,0)--(3,2,0)--(2,2,0)--(1,2,0)--(1,2,1)--(1,1,1)--(2,1,1)--(2,0,1)--(3,0,1)--(3,0,0)--cycle;
\draw[line width=1pt,fill=green]
(1,1,1)--(2,1,1)--(2,2,1)--(1,2,1)--cycle
;
\draw[line width=1pt,fill=yellow]
(2,1,1)--(3,1,1)--(3,2,1)--(2,2,1)--cycle
(2,0,1)--(3,0,1)--(3,1,1)--(2,1,1)--cycle
;
\draw[line width=1pt,fill=black!20]
(3,1,0)--(3,2,0)--(3,2,1)--(3,1,1)--cycle
(3,0,0)--(3,1,0)--(3,1,1)--(3,0,1)--cycle
;
\draw[line width=1pt,fill=black!50]
(2,2,0)--(2,2,1)--(3,2,1)--(3,2,0)--cycle
(1,2,0)--(1,2,1)--(2,2,1)--(2,2,0)--cycle
;
\draw[line width=2pt](3,0,0)--(3,1,0)--(3,2,0)--(2,2,0)--(1,2,0)--(1,2,1)--(1,1,1)--(2,1,1)--(2,0,1)--(3,0,1)--(3,0,0)--cycle;
\end{scope}
\begin{scope}[xshift=18cm,yshift=0cm]
\clip(3,0,0)--(3,1,0)--(2,1,0)--(2,1,1)--(2,0,1)--(3,0,1)--cycle;
\draw[line width=1pt,fill=black!5]
(2,0,1)--(3,0,1)--(3,1,1)--(2,1,1)--cycle
;
\draw[line width=1pt,fill=black!20]
(3,0,0)--(3,1,0)--(3,1,1)--(3,0,1)--cycle
;
\draw[line width=1pt,fill=black!50]
(2,1,0)--(2,1,1)--(3,1,1)--(3,1,0)--cycle
;
\draw[line width=2pt](3,0,0)--(3,1,0)--(2,1,0)--(2,1,1)--(2,0,1)--(3,0,1)--cycle;
\end{scope}
\begin{scope}[xshift=0cm,yshift=-44mm]
\begin{scope}[xshift=0cm,yshift=0cm]
\clip(3,0,0)--(3,4,0)--(0,4,0)--(0,3,0)--(1,3,0)--(1,2,0)--(1,1,0)--(2,1,0)--(2,0,0)--(3,0,0)--cycle;
\draw[line width=1pt,fill=black!5]
(2,3,0)--(3,3,0)--(3,4,0)--(2,4,0)--cycle
(2,2,0)--(3,2,0)--(3,3,0)--(2,3,0)--cycle
(2,1,0)--(3,1,0)--(3,2,0)--(2,2,0)--cycle
(2,0,0)--(3,0,0)--(3,1,0)--(2,1,0)--cycle
(1,3,0)--(2,3,0)--(2,4,0)--(1,4,0)--cycle
(1,2,0)--(2,2,0)--(2,3,0)--(1,3,0)--cycle
(1,1,0)--(2,1,0)--(2,2,0)--(1,2,0)--cycle
(0,3,0)--(1,3,0)--(1,4,0)--(0,4,0)--cycle
;
\draw[line width=1pt,fill=black!20]
;
\draw[line width=1pt,fill=black!50]
;
\draw[line width=2pt](3,0,0)--(3,4,0)--(0,4,0)--(0,3,0)--(1,3,0)--(1,2,0)--(1,1,0)--(2,1,0)--(2,0,0)--(3,0,0)--cycle;
\end{scope}
\begin{scope}[xshift=6cm,yshift=0cm]
\clip(3,0,0)--(3,4,0)--(0,4,0)--(0,4,1)--(0,3,1)--(1,3,1)--(1,2,1)--(1,1,1)--(2,1,1)--(2,0,1)--(3,0,1)--(3,0,0)--cycle;
\draw[line width=1pt,fill=black!5]
(2,3,0)--(3,3,0)--(3,4,0)--(2,4,0)--cycle
(2,2,0)--(3,2,0)--(3,3,0)--(2,3,0)--cycle
(2,1,1)--(3,1,1)--(3,2,1)--(2,2,1)--cycle
(2,0,1)--(3,0,1)--(3,1,1)--(2,1,1)--cycle
(1,3,1)--(2,3,1)--(2,4,1)--(1,4,1)--cycle
(1,2,1)--(2,2,1)--(2,3,1)--(1,3,1)--cycle
(1,1,1)--(2,1,1)--(2,2,1)--(1,2,1)--cycle
(0,3,1)--(1,3,1)--(1,4,1)--(0,4,1)--cycle
;
\draw[line width=1pt,fill=black!20]
(3,1,0)--(3,2,0)--(3,2,1)--(3,1,1)--cycle
(3,0,0)--(3,1,0)--(3,1,1)--(3,0,1)--cycle
(2,3,0)--(2,4,0)--(2,4,1)--(2,3,1)--cycle
(2,2,0)--(2,3,0)--(2,3,1)--(2,2,1)--cycle
;
\draw[line width=1pt,fill=black!50]
(2,2,0)--(2,2,1)--(3,2,1)--(3,2,0)--cycle
(1,4,0)--(1,4,1)--(2,4,1)--(2,4,0)--cycle
(0,4,0)--(0,4,1)--(1,4,1)--(1,4,0)--cycle
;
\draw[line width=2pt](3,0,0)--(3,4,0)--(0,4,0)--(0,4,1)--(0,3,1)--(1,3,1)--(1,2,1)--(1,1,1)--(2,1,1)--(2,0,1)--(3,0,1)--(3,0,0)--cycle;
\end{scope}
\begin{scope}[xshift=12cm,yshift=0cm]
\clip(3,0,0)--(3,4,0)--(0,4,0)--(0,4,1)--(0,3,1)--(1,3,1)--(1,3,2)--(1,2,2)--(1,1,2)--(2,1,2)--(2,1,1)--(2,0,1)--(3,0,1)--(3,0,0)--cycle;
\draw[line width=1pt,fill=black!5]
(2,3,0)--(3,3,0)--(3,4,0)--(2,4,0)--cycle
(2,2,0)--(3,2,0)--(3,3,0)--(2,3,0)--cycle
(2,1,1)--(3,1,1)--(3,2,1)--(2,2,1)--cycle
(2,0,1)--(3,0,1)--(3,1,1)--(2,1,1)--cycle
(1,3,1)--(2,3,1)--(2,4,1)--(1,4,1)--cycle
(1,2,2)--(2,2,2)--(2,3,2)--(1,3,2)--cycle
(1,1,2)--(2,1,2)--(2,2,2)--(1,2,2)--cycle
(0,3,1)--(1,3,1)--(1,4,1)--(0,4,1)--cycle
;
\draw[line width=1pt,fill=black!20]
(3,1,0)--(3,2,0)--(3,2,1)--(3,1,1)--cycle
(3,0,0)--(3,1,0)--(3,1,1)--(3,0,1)--cycle
(2,3,0)--(2,4,0)--(2,4,1)--(2,3,1)--cycle
(2,2,1)--(2,3,1)--(2,3,2)--(2,2,2)--cycle
(2,2,0)--(2,3,0)--(2,3,1)--(2,2,1)--cycle
(2,1,1)--(2,2,1)--(2,2,2)--(2,1,2)--cycle
;
\draw[line width=1pt,fill=black!50]
(2,2,0)--(2,2,1)--(3,2,1)--(3,2,0)--cycle
(1,4,0)--(1,4,1)--(2,4,1)--(2,4,0)--cycle
(1,3,1)--(1,3,2)--(2,3,2)--(2,3,1)--cycle
(0,4,0)--(0,4,1)--(1,4,1)--(1,4,0)--cycle
;
\draw[line width=2pt](3,0,0)--(3,4,0)--(0,4,0)--(0,4,1)--(0,3,1)--(1,3,1)--(1,3,2)--(1,2,2)--(1,1,2)--(2,1,2)--(2,1,1)--(2,0,1)--(3,0,1)--(3,0,0)--cycle;
\end{scope}
\begin{scope}[xshift=18cm,yshift=0cm]
\clip(3,0,0)--(3,4,0)--(0,4,0)--(0,4,1)--(0,3,1)--(1,3,1)--(1,3,2)--(1,2,2)--(1,1,2)--(2,1,2)--(2,0,2)--(3,0,2)--(3,0,0)--cycle;
\draw[line width=1pt,fill=black!5]
(2,3,0)--(3,3,0)--(3,4,0)--(2,4,0)--cycle
(1,3,1)--(2,3,1)--(2,4,1)--(1,4,1)--cycle
(1,2,2)--(2,2,2)--(2,3,2)--(1,3,2)--cycle
(1,1,2)--(2,1,2)--(2,2,2)--(1,2,2)--cycle
(0,3,1)--(1,3,1)--(1,4,1)--(0,4,1)--cycle
;
\draw[line width=1pt,fill=yellow]
(2,0,2)--(3,0,2)--(3,1,2)--(2,1,2)--cycle
(2,1,2)--(3,1,2)--(3,2,2)--(2,2,2)--cycle
;
\draw[line width=1pt,fill=green]
(2,2,1)--(3,2,1)--(3,3,1)--(2,3,1)--cycle
;
\draw[line width=1pt,fill=black!20]
(3,2,0)--(3,3,0)--(3,3,1)--(3,2,1)--cycle
(3,1,1)--(3,2,1)--(3,2,2)--(3,1,2)--cycle
(3,1,0)--(3,2,0)--(3,2,1)--(3,1,1)--cycle
(3,0,1)--(3,1,1)--(3,1,2)--(3,0,2)--cycle
(3,0,0)--(3,1,0)--(3,1,1)--(3,0,1)--cycle
(2,3,0)--(2,4,0)--(2,4,1)--(2,3,1)--cycle
(2,2,1)--(2,3,1)--(2,3,2)--(2,2,2)--cycle
;
\draw[line width=1pt,fill=black!50]
(2,3,0)--(2,3,1)--(3,3,1)--(3,3,0)--cycle
(2,2,1)--(2,2,2)--(3,2,2)--(3,2,1)--cycle
(1,4,0)--(1,4,1)--(2,4,1)--(2,4,0)--cycle
(1,3,1)--(1,3,2)--(2,3,2)--(2,3,1)--cycle
(0,4,0)--(0,4,1)--(1,4,1)--(1,4,0)--cycle
;
\draw[line width=2pt](3,0,0)--(3,4,0)--(0,4,0)--(0,4,1)--(0,3,1)--(1,3,1)--(1,3,2)--(1,2,2)--(1,1,2)--(2,1,2)--(2,0,2)--(3,0,2)--(3,0,0)--cycle;
\end{scope}
\end{scope}
\end{tikzpicture}
}
\end{align}

The examples \eqref{eq:0123-122-1(lex)} and \eqref{eq:0123-122-1(fac)} demonstrate that reverse plane partitions can in general be built in more than one way by successively inserting rim-hooks into the zero reverse plane partition.
However, \eqref{eq:0123-122-1(lex)} is the unique factorisation predicted by Theorem~\ref{Theorem:lexfact}, in which the building blocks are lexicographically increasing.

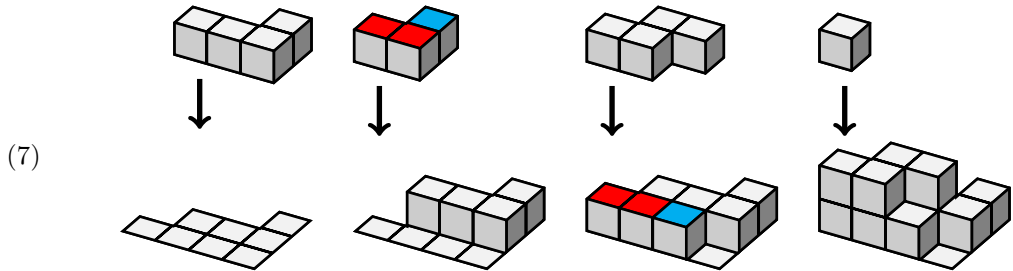
\begin{align}
\label{eq:0123-122-1(fac)}
\raisebox{-4em}{
\begin{tikzpicture}[scale=.6,tdplot_main_coords]
\begin{scope}[xshift=0cm,yshift=-2cm]
\draw[line width=2pt,->](1.5,1.5,1.5)--(1.5,1.5,0);
\begin{scope}[xshift=6cm,yshift=0cm]
\draw[line width=2pt,->](2.5,.5,1.5)--(2.5,.5,0);
\end{scope}
\begin{scope}[xshift=12cm,yshift=0cm]
\draw[line width=2pt,->](2.5,.5,1.5)--(2.5,.5,0);
\end{scope}
\begin{scope}[xshift=18cm,yshift=0cm]
\draw[line width=2pt,->](2.5,.5,1.5)--(2.5,.5,0);
\end{scope}
\end{scope}
\begin{scope}[xshift=0cm,yshift=0cm]
\clip(2,1,0)--(2,2,0)--(2,3,0)--(2,4,0)--(1,4,0)--(0,4,0)--(0,4,1)--(0,3,1)--(1,3,1)--(1,2,1)--(1,1,1)--(2,1,1)--cycle;
\draw[line width=1pt,fill=black!5]
(1,3,1)--(2,3,1)--(2,4,1)--(1,4,1)--cycle
(1,2,1)--(2,2,1)--(2,3,1)--(1,3,1)--cycle
(1,1,1)--(2,1,1)--(2,2,1)--(1,2,1)--cycle
(0,3,1)--(1,3,1)--(1,4,1)--(0,4,1)--cycle
;
\draw[line width=1pt,fill=black!20]
(2,3,0)--(2,4,0)--(2,4,1)--(2,3,1)--cycle
(2,2,0)--(2,3,0)--(2,3,1)--(2,2,1)--cycle
(2,1,0)--(2,2,0)--(2,2,1)--(2,1,1)--cycle
;
\draw[line width=1pt,fill=black!50]
(1,4,0)--(1,4,1)--(2,4,1)--(2,4,0)--cycle
(0,4,0)--(0,4,1)--(1,4,1)--(1,4,0)--cycle
;
\draw[line width=2pt](2,1,0)--(2,2,0)--(2,3,0)--(2,4,0)--(1,4,0)--(0,4,0)--(0,4,1)--(0,3,1)--(1,3,1)--(1,2,1)--(1,1,1)--(2,1,1)--cycle;
\end{scope}
\begin{scope}[xshift=6cm,yshift=0cm]
\clip(3,0,0)--(3,1,0)--(3,2,0)--(2,2,0)--(1,2,0)--(1,2,1)--(1,1,1)--(2,1,1)--(2,0,1)--(3,0,1)--(3,0,0)--cycle;
\draw[line width=1pt,fill=cyan]
(1,1,1)--(2,1,1)--(2,2,1)--(1,2,1)--cycle
;
\draw[line width=1pt,fill=red]
(2,1,1)--(3,1,1)--(3,2,1)--(2,2,1)--cycle
(2,0,1)--(3,0,1)--(3,1,1)--(2,1,1)--cycle
;
\draw[line width=1pt,fill=black!20]
(3,1,0)--(3,2,0)--(3,2,1)--(3,1,1)--cycle
(3,0,0)--(3,1,0)--(3,1,1)--(3,0,1)--cycle
;
\draw[line width=1pt,fill=black!50]
(2,2,0)--(2,2,1)--(3,2,1)--(3,2,0)--cycle
(1,2,0)--(1,2,1)--(2,2,1)--(2,2,0)--cycle
;
\draw[line width=2pt](3,0,0)--(3,1,0)--(3,2,0)--(2,2,0)--(1,2,0)--(1,2,1)--(1,1,1)--(2,1,1)--(2,0,1)--(3,0,1)--(3,0,0)--cycle;
\end{scope}
\begin{scope}[xshift=12cm,yshift=0cm]
\clip(3,0,0)--(3,1,0)--(3,2,0)--(2,2,0)--(2,3,0)--(1,3,0)--(1,3,1)--(1,2,1)--(1,1,1)--(2,1,1)--(2,0,1)--(3,0,1)--(3,0,0)--cycle;
\draw[line width=1pt,fill=black!5]
(2,1,1)--(3,1,1)--(3,2,1)--(2,2,1)--cycle
(2,0,1)--(3,0,1)--(3,1,1)--(2,1,1)--cycle
(1,2,1)--(2,2,1)--(2,3,1)--(1,3,1)--cycle
(1,1,1)--(2,1,1)--(2,2,1)--(1,2,1)--cycle
;
\draw[line width=1pt,fill=black!20]
(3,1,0)--(3,2,0)--(3,2,1)--(3,1,1)--cycle
(3,0,0)--(3,1,0)--(3,1,1)--(3,0,1)--cycle
(2,2,0)--(2,3,0)--(2,3,1)--(2,2,1)--cycle
;
\draw[line width=1pt,fill=black!50]
(2,2,0)--(2,2,1)--(3,2,1)--(3,2,0)--cycle
(1,3,0)--(1,3,1)--(2,3,1)--(2,3,0)--cycle
;
\draw[line width=2pt](3,0,0)--(3,1,0)--(3,2,0)--(2,2,0)--(2,3,0)--(1,3,0)--(1,3,1)--(1,2,1)--(1,1,1)--(2,1,1)--(2,0,1)--(3,0,1)--(3,0,0)--cycle;
\end{scope}
\begin{scope}[xshift=18cm,yshift=0cm]
\clip(3,0,0)--(3,1,0)--(2,1,0)--(2,1,1)--(2,0,1)--(3,0,1)--cycle;
\draw[line width=1pt,fill=black!5]
(2,0,1)--(3,0,1)--(3,1,1)--(2,1,1)--cycle
;
\draw[line width=1pt,fill=black!20]
(3,0,0)--(3,1,0)--(3,1,1)--(3,0,1)--cycle
;
\draw[line width=1pt,fill=black!50]
(2,1,0)--(2,1,1)--(3,1,1)--(3,1,0)--cycle
;
\draw[line width=2pt](3,0,0)--(3,1,0)--(2,1,0)--(2,1,1)--(2,0,1)--(3,0,1)--cycle;
\end{scope}
\begin{scope}[xshift=0cm,yshift=-44mm]
\begin{scope}[xshift=0cm,yshift=0cm]
\clip(3,0,0)--(3,4,0)--(0,4,0)--(0,3,0)--(1,3,0)--(1,2,0)--(1,1,0)--(2,1,0)--(2,0,0)--(3,0,0)--cycle;
\draw[line width=1pt,fill=black!5]
(2,3,0)--(3,3,0)--(3,4,0)--(2,4,0)--cycle
(2,2,0)--(3,2,0)--(3,3,0)--(2,3,0)--cycle
(2,1,0)--(3,1,0)--(3,2,0)--(2,2,0)--cycle
(2,0,0)--(3,0,0)--(3,1,0)--(2,1,0)--cycle
(1,3,0)--(2,3,0)--(2,4,0)--(1,4,0)--cycle
(1,2,0)--(2,2,0)--(2,3,0)--(1,3,0)--cycle
(1,1,0)--(2,1,0)--(2,2,0)--(1,2,0)--cycle
(0,3,0)--(1,3,0)--(1,4,0)--(0,4,0)--cycle
;
\draw[line width=1pt,fill=black!20]
;
\draw[line width=1pt,fill=black!50]
;
\draw[line width=2pt](3,0,0)--(3,4,0)--(0,4,0)--(0,3,0)--(1,3,0)--(1,2,0)--(1,1,0)--(2,1,0)--(2,0,0)--(3,0,0)--cycle;
\end{scope}
\begin{scope}[xshift=6cm,yshift=0cm]
\clip(3,0,0)--(3,4,0)--(0,4,0)--(0,4,1)--(0,3,1)--(1,3,1)--(1,2,1)--(1,1,1)--(2,1,1)--(2,1,0)--(2,0,0)--(3,0,0)--cycle;
\draw[line width=1pt,fill=black!5]
(2,3,0)--(3,3,0)--(3,4,0)--(2,4,0)--cycle
(2,2,0)--(3,2,0)--(3,3,0)--(2,3,0)--cycle
(2,1,0)--(3,1,0)--(3,2,0)--(2,2,0)--cycle
(2,0,0)--(3,0,0)--(3,1,0)--(2,1,0)--cycle
(1,3,1)--(2,3,1)--(2,4,1)--(1,4,1)--cycle
(1,2,1)--(2,2,1)--(2,3,1)--(1,3,1)--cycle
(1,1,1)--(2,1,1)--(2,2,1)--(1,2,1)--cycle
(0,3,1)--(1,3,1)--(1,4,1)--(0,4,1)--cycle
;
\draw[line width=1pt,fill=black!20]
(2,3,0)--(2,4,0)--(2,4,1)--(2,3,1)--cycle
(2,2,0)--(2,3,0)--(2,3,1)--(2,2,1)--cycle
(2,1,0)--(2,2,0)--(2,2,1)--(2,1,1)--cycle
;
\draw[line width=1pt,fill=black!50]
(1,4,0)--(1,4,1)--(2,4,1)--(2,4,0)--cycle
(0,4,0)--(0,4,1)--(1,4,1)--(1,4,0)--cycle
;
\draw[line width=2pt](3,0,0)--(3,4,0)--(0,4,0)--(0,4,1)--(0,3,1)--(1,3,1)--(1,2,1)--(1,1,1)--(2,1,1)--(2,1,0)--(2,0,0)--(3,0,0)--cycle;
\end{scope}
\begin{scope}[xshift=12cm,yshift=0cm]
\clip(3,0,0)--(3,4,0)--(0,4,0)--(0,4,1)--(0,3,1)--(1,3,1)--(1,2,1)--(1,1,1)--(2,1,1)--(2,0,1)--(3,0,1)--(3,0,0)--cycle;
\draw[line width=1pt,fill=black!5]
(2,3,0)--(3,3,0)--(3,4,0)--(2,4,0)--cycle
(1,3,1)--(2,3,1)--(2,4,1)--(1,4,1)--cycle
(1,2,1)--(2,2,1)--(2,3,1)--(1,3,1)--cycle
(1,1,1)--(2,1,1)--(2,2,1)--(1,2,1)--cycle
(0,3,1)--(1,3,1)--(1,4,1)--(0,4,1)--cycle
;
\draw[line width=1pt,fill=red]
(2,0,1)--(3,0,1)--(3,1,1)--(2,1,1)--cycle
(2,1,1)--(3,1,1)--(3,2,1)--(2,2,1)--cycle
;
\draw[line width=1pt,fill=cyan]
(2,2,1)--(3,2,1)--(3,3,1)--(2,3,1)--cycle
;
\draw[line width=1pt,fill=black!20]
(3,2,0)--(3,3,0)--(3,3,1)--(3,2,1)--cycle
(3,1,0)--(3,2,0)--(3,2,1)--(3,1,1)--cycle
(3,0,0)--(3,1,0)--(3,1,1)--(3,0,1)--cycle
(2,3,0)--(2,4,0)--(2,4,1)--(2,3,1)--cycle
;
\draw[line width=1pt,fill=black!50]
(2,3,0)--(2,3,1)--(3,3,1)--(3,3,0)--cycle
(1,4,0)--(1,4,1)--(2,4,1)--(2,4,0)--cycle
(0,4,0)--(0,4,1)--(1,4,1)--(1,4,0)--cycle
;
\draw[line width=2pt](3,0,0)--(3,4,0)--(0,4,0)--(0,4,1)--(0,3,1)--(1,3,1)--(1,2,1)--(1,1,1)--(2,1,1)--(2,0,1)--(3,0,1)--(3,0,0)--cycle;
\end{scope}
\begin{scope}[xshift=18cm,yshift=0cm]
\clip(3,0,0)--(3,4,0)--(0,4,0)--(0,4,1)--(0,3,1)--(1,3,1)--(1,3,2)--(1,2,2)--(1,1,2)--(2,1,2)--(2,0,2)--(3,0,2)--(3,0,0)--cycle;
\draw[line width=1pt,fill=black!5]
(2,3,0)--(3,3,0)--(3,4,0)--(2,4,0)--cycle
(2,2,1)--(3,2,1)--(3,3,1)--(2,3,1)--cycle
(2,1,2)--(3,1,2)--(3,2,2)--(2,2,2)--cycle
(2,0,2)--(3,0,2)--(3,1,2)--(2,1,2)--cycle
(1,3,1)--(2,3,1)--(2,4,1)--(1,4,1)--cycle
(1,2,2)--(2,2,2)--(2,3,2)--(1,3,2)--cycle
(1,1,2)--(2,1,2)--(2,2,2)--(1,2,2)--cycle
(0,3,1)--(1,3,1)--(1,4,1)--(0,4,1)--cycle
;
\draw[line width=1pt,fill=black!20]
(3,2,0)--(3,3,0)--(3,3,1)--(3,2,1)--cycle
(3,1,1)--(3,2,1)--(3,2,2)--(3,1,2)--cycle
(3,1,0)--(3,2,0)--(3,2,1)--(3,1,1)--cycle
(3,0,1)--(3,1,1)--(3,1,2)--(3,0,2)--cycle
(3,0,0)--(3,1,0)--(3,1,1)--(3,0,1)--cycle
(2,3,0)--(2,4,0)--(2,4,1)--(2,3,1)--cycle
(2,2,1)--(2,3,1)--(2,3,2)--(2,2,2)--cycle
;
\draw[line width=1pt,fill=black!50]
(2,3,0)--(2,3,1)--(3,3,1)--(3,3,0)--cycle
(2,2,1)--(2,2,2)--(3,2,2)--(3,2,1)--cycle
(1,4,0)--(1,4,1)--(2,4,1)--(2,4,0)--cycle
(1,3,1)--(1,3,2)--(2,3,2)--(2,3,1)--cycle
(0,4,0)--(0,4,1)--(1,4,1)--(1,4,0)--cycle
;
\draw[line width=2pt](3,0,0)--(3,4,0)--(0,4,0)--(0,4,1)--(0,3,1)--(1,3,1)--(1,3,2)--(1,2,2)--(1,1,2)--(2,1,2)--(2,0,2)--(3,0,2)--(3,0,0)--cycle;
\end{scope}
\end{scope}
\end{tikzpicture}
}
\end{align}

Using the bijection $\Psi$ it is elementary to obtain the trace generating function. 

\begin{proof}[Proof of Theorem~\ref{Thm:trace}]
For each rim-hook $h$ the set of contents $\{c(u):u\in h\}$ is an interval $\{j-\lambda_j',\dots,\lambda_i-i\}$ where $(i,j)$ is the cell corresponding to $h$.
Moreover if $P$ is a North-East-path in $\lambda$ with $\omega(P)=\omega(h)$ and $\ell(P)=\ell(h)$ then $\{c(u):u\in P\}=\{c(u):u\in h\}$.
Hence if $h$ inserts into a reverse plane partition $\pi$ then
\begin{align*}
\prod_{k\in\Z}q_k^{\tr_k(h*\pi)}
&=\prod_{k\in\Z}q_k^{\tr_k(\pi)}\cdot\prod_{k=j-\lambda_j'}^{\lambda_i-i}q_k\,.
\qedhere
\end{align*}
\end{proof}

\section{The inverse map}\label{Section:factors}

This section treats the inverse of the map $\Psi$ defined in Section~\ref{Section:insert}.
That is, given a reverse plane partition how do we find its lexicographic factorisation?
An answer to this question is found in Theorem~\ref{Theorem:inverse}, which provides an inductive algorithm for constructing the lexicographic factorisation.

\smallskip
First consider a simpler problem.
A rim-hook $h$ is a \emph{factor} of a reverse plane partition $\pi$ if there exists a reverse plane partition $\tilde{\pi}$ such that $\pi=h*\tilde{\pi}$.
Note that the reverse plane partition $\tilde{\pi}$ is not unique in general, as is demonstrated in the example below.

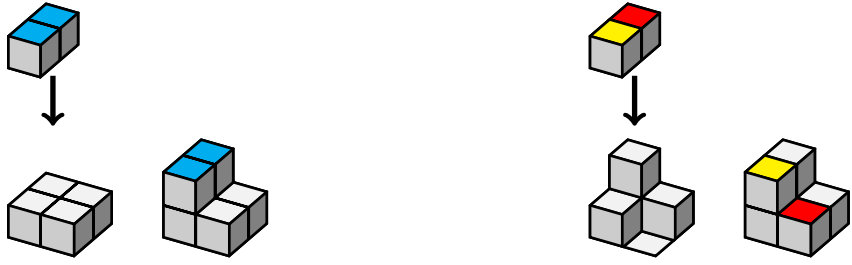
\begin{align*}
\begin{tikzpicture}[scale=.6,tdplot_main_coords]
\begin{scope}[xshift=0cm,yshift=-2cm]
\draw[line width=2pt,->](.5,.5,1.5)--(.5,.5,0);
\end{scope}
\begin{scope}[xshift=0cm,yshift=0cm]
\clip(2,0,0)--(2,1,0)--(0,1,0)--(0,1,1)--(0,0,1)--(2,0,1)--cycle;
\draw[line width=1pt,fill=cyan]
(0,0,1)--(1,0,1)--(1,1,1)--(0,1,1)--cycle
(1,0,1)--(2,0,1)--(2,1,1)--(1,1,1)--cycle
;
\draw[line width=1pt,fill=black!20]
(2,0,1)--(2,0,0)--(2,1,0)--(2,1,1)--cycle
;
\draw[line width=1pt,fill=black!50]
(0,1,1)--(1,1,1)--(1,1,0)--(0,1,0)--cycle
(1,1,1)--(2,1,1)--(2,1,0)--(1,1,0)--cycle
;
\draw[line width=2pt](2,0,0)--(2,1,0)--(0,1,0)--(0,1,1)--(0,0,1)--(2,0,1)--cycle;
\end{scope}
\begin{scope}[xshift=0cm,yshift=-44mm]
\clip(2,0,0)--(2,2,0)--(0,2,0)--(0,2,1)--(0,1,1)--(0,1,1)--(0,0,1)--(2,0,1)--cycle;
\draw[line width=1pt,fill=black!5]
(0,0,1)--(1,0,1)--(1,1,1)--(0,1,1)--cycle
(0,1,1)--(1,1,1)--(1,2,1)--(0,2,1)--cycle
(1,0,1)--(2,0,1)--(2,1,1)--(1,1,1)--cycle
(1,1,1)--(2,1,1)--(2,2,1)--(1,2,1)--cycle
;
\draw[line width=1pt,fill=black!20]
(2,0,1)--(2,0,0)--(2,1,0)--(2,1,1)--cycle
(2,1,1)--(2,1,0)--(2,2,0)--(2,2,1)--cycle
;
\draw[line width=1pt,fill=black!50]
(0,2,1)--(1,2,1)--(1,2,0)--(0,2,0)--cycle
(1,2,1)--(2,2,1)--(2,2,0)--(1,2,0)--cycle
;
\draw[line width=2pt](2,0,0)--(2,2,0)--(0,2,0)--(0,2,1)--(0,1,1)--(0,1,1)--(0,0,1)--(2,0,1)--cycle;
\end{scope}
\begin{scope}[xshift=4cm,yshift=-44mm]
\clip(2,0,0)--(2,2,0)--(0,2,0)--(0,2,1)--(0,1,1)--(0,1,2)--(0,0,2)--(2,0,2)--cycle;
\draw[line width=1pt,fill=black!5]
(0,1,1)--(1,1,1)--(1,2,1)--(0,2,1)--cycle
(1,1,1)--(2,1,1)--(2,2,1)--(1,2,1)--cycle
;
\draw[line width=1pt,fill=cyan]
(0,0,2)--(1,0,2)--(1,1,2)--(0,1,2)--cycle
(1,0,2)--(2,0,2)--(2,1,2)--(1,1,2)--cycle
;
\draw[line width=1pt,fill=black!20]
(2,0,2)--(2,0,1)--(2,1,1)--(2,1,2)--cycle
(2,0,1)--(2,0,0)--(2,1,0)--(2,1,1)--cycle
(2,1,1)--(2,1,0)--(2,2,0)--(2,2,1)--cycle
;
\draw[line width=1pt,fill=black!50]
(0,1,2)--(1,1,2)--(1,1,1)--(0,1,1)--cycle
(0,2,1)--(1,2,1)--(1,2,0)--(0,2,0)--cycle
(1,1,2)--(2,1,2)--(2,1,1)--(1,1,1)--cycle
(1,2,1)--(2,2,1)--(2,2,0)--(1,2,0)--cycle
;
\draw[line width=2pt](2,0,0)--(2,2,0)--(0,2,0)--(0,2,1)--(0,1,1)--(0,1,2)--(0,0,2)--(2,0,2)--cycle;
\end{scope}
\begin{scope}[xshift=15cm]
\begin{scope}[xshift=0cm,yshift=-2cm]
\draw[line width=2pt,->](.5,.5,1.5)--(.5,.5,0);
\end{scope}
\begin{scope}[xshift=0cm,yshift=0cm]
\clip(2,0,0)--(2,1,0)--(0,1,0)--(0,1,1)--(0,0,1)--(2,0,1)--cycle;
\draw[line width=1pt,fill=red]
(0,0,1)--(1,0,1)--(1,1,1)--(0,1,1)--cycle
;
\draw[line width=1pt,fill=yellow]
(1,0,1)--(2,0,1)--(2,1,1)--(1,1,1)--cycle
;
\draw[line width=1pt,fill=black!20]
(2,0,1)--(2,0,0)--(2,1,0)--(2,1,1)--cycle
;
\draw[line width=1pt,fill=black!50]
(0,1,1)--(1,1,1)--(1,1,0)--(0,1,0)--cycle
(1,1,1)--(2,1,1)--(2,1,0)--(1,1,0)--cycle
;
\draw[line width=2pt](2,0,0)--(2,1,0)--(0,1,0)--(0,1,1)--(0,0,1)--(2,0,1)--cycle;
\end{scope}
\begin{scope}[xshift=0cm,yshift=-44mm]
\clip(2,0,0)--(2,2,0)--(0,2,0)--(0,2,1)--(0,1,1)--(0,1,2)--(0,0,2)--(1,0,2)--(1,0,1)--(2,0,1)--cycle;
\draw[line width=1pt,fill=black!5]
(0,0,2)--(1,0,2)--(1,1,2)--(0,1,2)--cycle
(0,1,1)--(1,1,1)--(1,2,1)--(0,2,1)--cycle
(1,0,1)--(2,0,1)--(2,1,1)--(1,1,1)--cycle
(1,1,0)--(2,1,0)--(2,2,0)--(1,2,0)--cycle
;
\draw[line width=1pt,fill=black!20]
(1,0,2)--(1,0,1)--(1,1,1)--(1,1,2)--cycle
(1,1,1)--(1,1,0)--(1,2,0)--(1,2,1)--cycle
(2,0,1)--(2,0,0)--(2,1,0)--(2,1,1)--cycle
;
\draw[line width=1pt,fill=black!50]
(0,1,2)--(1,1,2)--(1,1,1)--(0,1,1)--cycle
(0,2,1)--(1,2,1)--(1,2,0)--(0,2,0)--cycle
(1,1,1)--(2,1,1)--(2,1,0)--(1,1,0)--cycle
;
\draw[line width=2pt](2,0,0)--(2,2,0)--(0,2,0)--(0,2,1)--(0,1,1)--(0,1,2)--(0,0,2)--(1,0,2)--(1,0,1)--(2,0,1)--cycle;
\end{scope}
\begin{scope}[xshift=4cm,yshift=-44mm]
\clip(2,0,0)--(2,2,0)--(0,2,0)--(0,2,1)--(0,1,1)--(0,1,2)--(0,0,2)--(2,0,2)--cycle;
\draw[line width=1pt,fill=black!5]
(0,0,2)--(1,0,2)--(1,1,2)--(0,1,2)--cycle
(0,1,1)--(1,1,1)--(1,2,1)--(0,2,1)--cycle
;
\draw[line width=1pt,fill=red]
(1,1,1)--(2,1,1)--(2,2,1)--(1,2,1)--cycle
;
\draw[line width=1pt,fill=yellow]
(1,0,2)--(2,0,2)--(2,1,2)--(1,1,2)--cycle
;
\draw[line width=1pt,fill=black!20]
(2,0,2)--(2,0,1)--(2,1,1)--(2,1,2)--cycle
(2,0,1)--(2,0,0)--(2,1,0)--(2,1,1)--cycle
(2,1,1)--(2,1,0)--(2,2,0)--(2,2,1)--cycle
;
\draw[line width=1pt,fill=black!50]
(0,1,2)--(1,1,2)--(1,1,1)--(0,1,1)--cycle
(0,2,1)--(1,2,1)--(1,2,0)--(0,2,0)--cycle
(1,1,2)--(2,1,2)--(2,1,1)--(1,1,1)--cycle
(1,2,1)--(2,2,1)--(2,2,0)--(1,2,0)--cycle
;
\draw[line width=2pt](2,0,0)--(2,2,0)--(0,2,0)--(0,2,1)--(0,1,1)--(0,1,2)--(0,0,2)--(2,0,2)--cycle;
\end{scope}
\end{scope}
\end{tikzpicture}
\end{align*}

In this sense the idea of finding a South-West-path $P$ starting at $\omega(h)$ similar to \eqref{eq:P} and reducing each entry of $\pi$ along $P$ is flawed.
Instead we shall work with North-East-paths.
Indeed it turns that $\tilde{\pi}$ is unique if one fixes $\alpha(P(\tilde{\pi},h))$.
The natural question arises, which cells $u\in\lambda$ are possible candidates for $u=\alpha(P(\tilde{\pi},h))$.
This leads to the following definition.

The set of \emph{candidates} of $\pi$ is defined as
\begin{align*}
\C(\pi)
=\big\{u\in\O:\pi(u)>\pi(\w u)\big\}
\cup\big\{u\in\A:\pi(u)>\pi(\w u)\text{ and }\pi(u)>\pi(\n u)\big\}.
\end{align*}
The example below shows the candidates of our favourite reverse plane partition.
For example, the cell $u=(3,1)$ is a candidate since $u\in\O$ and $\pi(u)>\pi(\w u)$, where $\pi(\w u)=0$ by convention.
The cell $u=(1,3)$ is not a candidate even though $\pi(u)>\pi(\w u)$ and $\pi(u)>\pi(\n u)$ because $u\notin\A\cup\O$.
\begin{align*}
\begin{tikzpicture}[scale=.6]
\begin{scope}[xshift=0cm]
\fill[yellow]
(3,2)--(4,2)--(4,3)--(3,3)
(1,2)--(2,2)--(2,3)--(1,3)
(1,1)--(2,1)--(2,2)--(1,2)
(0,0)--(1,0)--(1,1)--(0,1);
\draw
(0,1)--(1,1)
(0,2)--(3,2)
(1,1)--(1,3)
(2,1)--(2,3)
(3,2)--(3,3)
;
\draw[line width=1pt](0,0)--(1,0)--(1,1)--(3,1)--(3,2)--(4,2)--(4,3)--(0,3)--cycle;
\draw[xshift=5mm,yshift=5mm]
(0,2)node{{$0$}}
(1,2)node{{$1$}}
(2,2)node{{$2$}}
(3,2)node{{$3$}}
(0,1)node{{$1$}}
(1,1)node{{$2$}}
(2,1)node{{$2$}}
(0,0)node{{$1$}};
\end{scope}
\end{tikzpicture}
\end{align*}

Given a candidate cell $u\in\C(\pi)$ define a North-East-path $Q(\pi,u)$ in $\lambda$ as follows.
Set $\alpha(Q(\pi,u))=u$.
If $v$ is the current cell then move to
\begin{align*}
\text{the cell}\quad
\begin{cases}
\n v&\quad\text{if }v\in\O\cup\B\text{ and }\pi(\n v)=\pi(v),\\
\e v&\quad\text{if }v\in\I\cup\A\text{ or }\pi(\n v)<\pi(v)\text{ and }\e v\in\lambda,
\end{cases}
\end{align*}
and terminate the path if $\pi(\n v)<\pi(v)$ but $\e v\notin\lambda$.
Moreover, define a rim-hook $h(\pi,u)$ by demanding $\omega(h(\pi,u))=\omega(Q(\pi,u))$ and $\ell(h(\pi,u))=\ell(Q(\pi,u))$.

One can prove that $h$ is a factor of $\pi$ if and only if $h=h(\pi,u)$ for some candidate $u\in\C(\pi)$ such that $Q(\pi,u)$ is $\pi$-compatible and $\pi-Q(\pi,u)$ is a reverse plane partition.
To obtain a lexicographic factorisation of $\pi$ it therefore suffices to determine the correct candidate.
The correct candidate turns out to be the minimal candidate with respect to the content order.

\begin{thm}\label{Theorem:inverse} Let $\lambda$ be a partition, $\pi$ a reverse plane partition of shape $\lambda$ and let $u$ be the minimum of $\C(\pi)$ with respect to the content order.
Then a lexicographic factorisation is given by $h(\pi,u),h_2,\dots,h_s$, where $h_2,\dots,h_s$ is a lexicographic factorisation of $\pi-Q(\pi,u)$.
\end{thm}

For example, the lexicographic factorisation \eqref{eq:0123-122-1(lex)} of the reverse plane partition from \eqref{eq:0123-122-1} is obtained in \eqref{eq:candidates} below.
The candidates are highlighted in yellow, the minimal candidate $u$ and the path $Q(\pi,u)$ is highlighted in orange in each of the reverse plane partitions.

\begin{align}
\label{eq:candidates}
\raisebox{-2em}{
\begin{tikzpicture}[scale=.6]
\begin{scope}[xshift=0cm]
\fill[orange]
(3,2)--(4,2)--(4,3)--(3,3);
\fill[yellow]
(1,2)--(2,2)--(2,3)--(1,3)
(1,1)--(2,1)--(2,2)--(1,2)
(0,0)--(1,0)--(1,1)--(0,1);
\draw
(0,1)--(1,1)
(0,2)--(3,2)
(1,1)--(1,3)
(2,1)--(2,3)
(3,2)--(3,3)
;
\draw[line width=1pt](0,0)--(1,0)--(1,1)--(3,1)--(3,2)--(4,2)--(4,3)--(0,3)--cycle;
\draw[xshift=5mm,yshift=5mm]
(0,2)node{{$0$}}
(1,2)node{{$1$}}
(2,2)node{{$2$}}
(3,2)node{{$3$}}
(0,1)node{{$1$}}
(1,1)node{{$2$}}
(2,1)node{{$2$}}
(0,0)node{{$1$}};
\end{scope}
\begin{scope}[xshift=5.5cm]
\fill[orange]
(1,2)--(2,2)--(2,3)--(1,3);
\fill[yellow]
(1,1)--(2,1)--(2,2)--(1,2)
(0,0)--(1,0)--(1,1)--(0,1);
\draw
(0,1)--(1,1)
(0,2)--(3,2)
(1,1)--(1,3)
(2,1)--(2,3)
(3,2)--(3,3)
;
\draw[xshift=5mm,yshift=5mm,line width=2pt,orange](1,2)--(3.5,2);
\draw[line width=1pt](0,0)--(1,0)--(1,1)--(3,1)--(3,2)--(4,2)--(4,3)--(0,3)--cycle;
\draw[xshift=5mm,yshift=5mm]
(0,2)node{{$0$}}
(1,2)node{{$1$}}
(2,2)node{{$2$}}
(3,2)node{{$2$}}
(0,1)node{{$1$}}
(1,1)node{{$2$}}
(2,1)node{{$2$}}
(0,0)node{{$1$}};
\end{scope}
\begin{scope}[xshift=11cm]
\fill[orange]
(1,1)--(2,1)--(2,2)--(1,2);
\fill[yellow]
(0,0)--(1,0)--(1,1)--(0,1);
\draw
(0,1)--(1,1)
(0,2)--(3,2)
(1,1)--(1,3)
(2,1)--(2,3)
(3,2)--(3,3)
;
\draw[xshift=5mm,yshift=5mm,line width=2pt,orange](1,1)--(2.5,1);
\draw[line width=1pt](0,0)--(1,0)--(1,1)--(3,1)--(3,2)--(4,2)--(4,3)--(0,3)--cycle;
\draw[xshift=5mm,yshift=5mm]
(0,2)node{{$0$}}
(1,2)node{{$0$}}
(2,2)node{{$1$}}
(3,2)node{{$1$}}
(0,1)node{{$1$}}
(1,1)node{{$2$}}
(2,1)node{{$2$}}
(0,0)node{{$1$}};
\end{scope}
\begin{scope}[xshift=16.5cm]
\fill[orange]
(0,0)--(1,0)--(1,1)--(0,1);
\draw
(0,1)--(1,1)
(0,2)--(3,2)
(1,1)--(1,3)
(2,1)--(2,3)
(3,2)--(3,3)
;
\draw[xshift=5mm,yshift=5mm,line width=2pt,orange](0,0)--(0,1)--(2,1)--(2,2)--(3.5,2);
\draw[line width=1pt](0,0)--(1,0)--(1,1)--(3,1)--(3,2)--(4,2)--(4,3)--(0,3)--cycle;
\draw[xshift=5mm,yshift=5mm]
(0,2)node{{$0$}}
(1,2)node{{$0$}}
(2,2)node{{$1$}}
(3,2)node{{$1$}}
(0,1)node{{$1$}}
(1,1)node{{$1$}}
(2,1)node{{$1$}}
(0,0)node{{$1$}};
\end{scope}
\begin{scope}[xshift=22cm]
\draw
(0,1)--(1,1)
(0,2)--(3,2)
(1,1)--(1,3)
(2,1)--(2,3)
(3,2)--(3,3)
;
\draw[line width=1pt](0,0)--(1,0)--(1,1)--(3,1)--(3,2)--(4,2)--(4,3)--(0,3)--cycle;
\draw[xshift=5mm,yshift=5mm]
(0,2)node{{$0$}}
(1,2)node{{$0$}}
(2,2)node{{$0$}}
(3,2)node{{$0$}}
(0,1)node{{$0$}}
(1,1)node{{$0$}}
(2,1)node{{$0$}}
(0,0)node{{$0$}};
\end{scope}
\end{tikzpicture}
}
\end{align}

%

%

%


\bibliographystyle{amsalpha}
\bibliography{sample}

\end{document}